\documentclass[11pt, reqno]{amsart}
\usepackage{amsmath,amssymb,amsfonts,amsthm,graphicx,mathrsfs,mathtools}
\usepackage{array}
\usepackage{blkarray}
\usepackage[margin=1.5in]{geometry}
\usepackage[dvipsnames]{xcolor}
\usepackage{graphicx,tipa}
\usepackage{tikz}
\usepackage{graphicx,amsmath,calc}
\usepackage{subcaption}
\makeatletter
\@namedef{subjclassname@2020}{%
\textup{2020} Mathematics Subject Classification}
\makeatother

\title[On stationary real matrix Schubert varieties]{On stationary real matrix Schubert varieties}
\author[J. Lee]{Jaehoon Lee}
\address[]{Jaehoon Lee, School of Mathematics, Korea Institute for Advanced Study, 85 Hoegiro, Dongdaemun-gu, Seoul 02455, Republic of Korea}
\email{jaehoonlee@kias.re.kr}

\author[S. Park]{Sangwoo Park}
\address[]{Sangwoo Park, Department of Mathematical Sciences, Pusan National University, Busan 46241, Republic of Korea}
\email{s.w.park@pusan.ac.kr}

\author[E. Yeon]{Eungbeom Yeon}
\address[]{Eungbeom Yeon, Department of Mathematical Sciences, Pusan National University, Busan 46241, Republic of Korea}
\email{ebeom.yeon@pusan.ac.kr}

\allowdisplaybreaks

\begin{document}

\newtheorem{theorem}{theorem}[section]
\newtheorem{thm}[theorem]{Theorem}
\newtheorem{lemma}[theorem]{Lemma}
\newtheorem{cor}[theorem]{Corollary}
\newtheorem{prop}[theorem]{Proposition}
\newtheorem{rmk}[theorem]{Remark}
\newtheorem{Ex}[theorem]{Example}
\newtheorem{Def}[theorem]{Definition}
\newtheorem{Question}[theorem]{Question}
\newtheorem*{conj}{Conjecture}
\newtheorem*{mainthm1}{Theorem 1}
\newtheorem*{mainthm2}{Theorem 2}
\newtheorem*{mainthm3}{Theorem 3}

\renewcommand{\theequation}{\thesection.\arabic{equation}}
\newcommand{\RNum}[1]{\uppercase\expandafter{\romannumeral #1\relax}}
\newcommand{\R}{\mathbb{R}}
\newcommand{\C}{\mathbb{C}}
\newcommand{\grad}{\nabla}
\newcommand{\laplacian}{\Delta}
\newcommand{\tgamma}{\tilde{\gamma}}
\newcommand{\ttau}{\tilde{\tau}}
\newcommand{\pp}{\Phi}
\newcommand{\tkappa}{\tilde{\kappa}}
\renewcommand{\d}{\textup{d}}
\newlength{\mywidth}
\newcommand\bigfrown[2][\textstyle]{\ensuremath{%
  \array[b]{c}\text{\resizebox{\mywidth}{.7ex}{$#1\frown$}}\\[-1.3ex]#1#2\endarray}}
\newcommand{\arc}[1]{{%
  \setbox9=\hbox{#1}%
  \ooalign{\resizebox{\wd9}{\height}{\texttoptiebar{\phantom{A}}}\cr#1}}}
\newcommand\buildcirclepm[1]{%

  \begin{tikzpicture}[baseline=(X.base), inner sep=-0.1, outer sep=-2]

    \node[draw,circle] (X)  {\footnotesize\raisebox{1ex}{$#1\pm$}};

  \end{tikzpicture}%

}


\begin{abstract}
In this paper, we study when a real matrix Schubert variety is stationary with respect to the first variation. We first show that a necessary condition for its open dense regular part to be a minimal submanifold is that the corresponding partial permutation is vexillary. Among vexillary partial permutations, we establish minimality by a geometric argument when the Rothe diagram is of Grassmannian type and has at most two connected components. We further obtain, as a corollary, the minimality of those varieties that decompose as products of this type. These varieties include all determinantal varieties as well as some new minimal cones.
\end{abstract}

\maketitle

\section{Introduction}\label{intro}
\setcounter{equation}{0}
Many known area-minimizing cones are realized as real algebraic varieties. One of the most significant examples is the Simons cone, which is given by the zero set of the polynomial 
\begin{align*}
x_1^2+x_2^2+x_3^2+x_4^2-x_5^2-x_6^2-x_7^2-x_8^2.
\end{align*} 
Its explicit representation enabled various geometric observations and insights, and consequently it has played a central role in the regularity theory of minimal hypersurfaces. On the other hand, Hsiang \cite{Hsiang} observed that every homogeneous minimal submanifold of the round sphere must be a real algebraic variety. Here, a minimal submanifold is said to be homogeneous if it is an orbit of a subgroup of the isometry group of the ambient Riemannian manifold. 

These results suggest that real algebraic varieties form a promising class, providing not only area-minimizing cones but also various minimal submanifolds that can be studied explicitly. However, algebraic varieties may contain singularities, and it is common that the numbers of generators of their defining ideals are greater than their codimensions, making it difficult to handle Riemannian geometric quantities. Therefore a natural direction is to extend observations from examples that are already well understood.

Recently, Bordemann, Choe and Hoppe \cite{BCH}, and independently, Kozhasov \cite{Koz} established that all determinantal varieties give rise to minimal submanifolds. Given $m$, $n$, and $0<r<\min\{m,n\}$, the set of all real $m\times n$ matrices of rank at most $r$,
\begin{align*}
\overline{C}(m,n,r):=\left\{A\in\mathfrak{M}_{m,n}(\mathbb{R})\mid \text{rk}(A)\leq r\right\},
\end{align*}
is called the determinantal variety. Here, $\mathfrak{M}_{m,n}(\mathbb{R})$ denotes the space of all real $m\times n$ matrices equipped with the inner product $\left\langle A, B\right\rangle :=\text{tr}\left(A^TB\right)$. This is a real algebraic variety defined by the vanishing of all $(r+1)\times (r+1)$ minors. It contains an open dense regular subset $C(m,n,r)$ consisting of matrices of rank exactly $r$. In \cite{BCH} and \cite{Koz}, several different methods were applied to prove that every $C(m,n,r)$ is a minimal submanifold in $\mathfrak{M}_{m,n}(\mathbb{R})$.

In \cite{KerLaw}, Kerckhove and Lawlor studied the area-minimizing property of $\overline{C}(m,n,r)$ using the directed slicing method and proved that it is area-minimizing whenever $m+n-2r\geq 4$. Similarly, the minimality of all Pfaffian varieties was also proved in \cite{Koz}, and further area-minimizing cones were identified by Cui, Jiao and Xu \cite{CJX}.

In this paper, we study the minimality of real matrix Schubert varieties introduced in \cite{Ful}. For a partial permutation $\omega\in\mathfrak{M}_{m,n}(\mathbb{R})$ (see Definition \ref{Def21}), the real matrix Schubert variety $\overline{X}_{\omega}$ is defined by
\begin{align*}
\overline{X}_{\omega}:=\left\{A\in\mathfrak{M}_{m,n}(\mathbb{R})\mid \text{rk}\left(A_{[p,q]}\right)\leq\text{rk}\left(\omega_{[p,q]}\right)\text{ for all }1\leq p\leq m, 1\leq q\leq n\right\},
\end{align*}
where $A_{[p,q]}$ denotes the upper-left $p\times q$ submatrix of $A$. As mentioned in Example \ref{ex27}, these varieties naturally generalize determinantal varieties. The real matrix Schubert variety $\overline{X}_\omega$ also contains an open dense regular subset $X_\omega$ given by
\begin{align*}
X_{\omega}:=\left\{A\in\mathfrak{M}_{m,n}(\mathbb{R})\mid \text{rk}\left(A_{[p,q]}\right)=\text{rk}\left(\omega_{[p,q]}\right)\text{ for all }1\leq p\leq m, 1\leq q\leq n\right\}.
\end{align*}
For further details on matrix Schubert varieties, see Section \ref{schubert} or \cite{Ful}.

We first establish the following necessary condition for the open dense regular part $X_\omega$ to be a minimal submanifold in $\mathfrak{M}_{m,n}(\mathbb{R})$.
\begin{mainthm1}[Theorem \ref{Thm41}]
If $\omega$ is a non-vexillary partial permutation, then $X_{\omega}\subset\mathfrak{M}_{m,n}(\mathbb{R})$ is not minimal.
\end{mainthm1}
Those permutations that avoid the $2143$-pattern are called \emph{vexillary}. The precise definition for vexillary partial permutations is given in Definition \ref{Def25}, and an equivalent characterization appears in Lemma \ref{vex}. The idea of the proof is as follows. We first construct a normal frame for $X_\omega$ in Proposition \ref{Prop26}. If $\omega$ is not vexillary, we then use this pattern to construct the point $\widehat{\omega}^{(\alpha,\beta)}(y_1,y_2,y_3)$ (see (\ref{omegappp})), which lies in $X_\omega$. Using the normal frame, we show that the mean curvature vector does not vanish at the point $\widehat{\omega}^{(\alpha,\beta)}(y_1,y_2,y_3)$. See Section \ref{nonminimal} for more details.

Vexillary permutations in the symmetric group were first identified by Lascoux and Sch\"{u}tzenberger \cite{LS} through pattern avoidance. They admit several equivalent characterizations, which connect various concepts in algebraic combinatorics and algebraic geometry. Therefore vexillary permutations in the symmetric group have played a central role in these areas, in part due to their characterizations serving as starting points for introducing the notion of vexillary elements in other settings. Billey and Lam \cite{BL} extended the notion to other Lie types via Stanley symmetric functions. Fulton \cite{Ful} established a geometric interpretation through the study of degeneracy loci, and this viewpoint was further developed by Anderson and Fulton, extending it to vexillary signed permutations (see \cite{AF} and references therein).

The above theorem suggests that vexillary permutations, which have been studied mainly from combinatorial and algebraic perspectives, may also have analytic implications through the mean-curvature geometry of Schubert varieties. We expect that the vexillary condition should also be sufficient for minimality.
\begin{conj}
If $\omega$ is a vexillary partial permutation, then $X_\omega\subset\mathfrak{M}_{m,n}(\mathbb{R})$ is minimal.
\end{conj}
Since determinantal varieties correspond to the vexillary case, the conjecture is verified in this special case, while it remains widely open for other cases. We confirm this conjecture for a certain subclass $\widetilde{\mathfrak{Gr}}_2$, which consists of partial permutations whose Rothe diagrams have the same shape as those of Grassmannian permutations and have at most two connected components. Here, a permutation is called Grassmannian if it has just one descent.
\begin{mainthm2}[Theorem \ref{thm52}]
If $\omega\in\widetilde{\mathfrak{Gr}}_2$, then $X_{\omega}\subset\mathfrak{M}_{m,n}(\mathbb{R})$ is minimal.
\end{mainthm2}
See Section \ref{grass} for the precise description of $\widetilde{\mathfrak{Gr}}_2$. The case in which the Rothe diagram has exactly two connected components yields new examples distinct from determinantal varieties. 

For the proof, we modify the geometric argument given in \cite{BCH}, where the authors proved the minimality of determinantal varieties by showing that they admit helicoidal symmetries. In our case the varieties are not helicoidal in general. Nevertheless, by considering analogous involutive isometries, we show that certain components of the mean curvature vector must vanish under the action of these isometries. By iterating this procedure, we eventually obtain minimality.

If the Rothe diagram of a vexillary partial permutation contains the top-left corner $(1,1)$, then the corresponding matrix Schubert variety decomposes as a product of smaller matrix Schubert varieties together with a Euclidean factor. As a corollary, we also deduce minimality in the following case.
\begin{mainthm3}[Corollary \ref{corcorcorcor}]
Let $\omega\in\mathfrak{M}_{m,n}(\mathbb{R})$ be a vexillary partial permutation whose Rothe diagram contains the top-left corner $(1,1)$. Then $\overline{X}_{\omega}$ is congruent to $\overline{X}_{\omega_1}\times\cdots\times\overline{X}_{\omega_k}\times\mathbb{R}^N$ for some $N\geq 0$ and vexillary partial permutations $\omega_1,\cdots,\omega_k$. If all $\omega_j$ belong to $\widetilde{\mathfrak{Gr}}_2$, then $X_{\omega}\subset\mathfrak{M}_{m,n}(\mathbb{R})$ is minimal.
\end{mainthm3}

The paper is organized as follows. In Section \ref{schubert}, we review the necessary background on real matrix Schubert varieties, including the notions of vexillary partial permutations and the Rothe diagram. Then we construct a normal frame for their open dense regular parts. Section \ref{lemmaminor} provides computational lemmas on minors that are used in the proof. In Section \ref{nonminimal}, we prove the non-minimality of matrix Schubert varieties corresponding to non-vexillary partial permutations, which gives a necessary condition for minimality. Section \ref{grass} defines the subclass $\widetilde{\mathfrak{Gr}}_2$ of vexillary partial permutations and establishes the minimality through a geometric argument. We also discuss the case when the Rothe diagram of a vexillary partial permutation contains the top-left corner, in which the variety decomposes as a product of smaller matrix Schubert varieties together with a Euclidean factor.

\section*{Acknowledgements}
JL was supported by a KIAS Individual Grant MG086402 at Korea Institute for Advanced Study. SP was supported by the National Research Foundation of Korea (NRF) grant funded by the Korea government (MSIT) (No. RS-2023-00245494). EY was supported in part by National Research Foundation of Korea NRF-2022R1C1C2013384.

\section{Real matrix Schubert varieties}\label{schubert}
\setcounter{equation}{0}
This section provides a brief review of partial permutations and the associated real matrix Schubert varieties. See \cite{Ful, combcommalg} for further details. We then consider the geometric properties of real matrix Schubert varieties with particular emphasis on the mean curvature vector. 

\subsection{Partial permutations and the Rothe diagram}\label{ppermu}
Let $\mathfrak{M}_{m,n}(\mathbb{R})$ be the $\mathbb{R}$-vector space of all real $m \times n$ matrices, equipped with the inner product $\left\langle A, B\right\rangle=\text{tr}\left(A^T B\right)$. This space is isometric to the Euclidean space $\mathbb{R}^{mn}$. 

\begin{Def}\label{Def21}\normalfont
A matrix $\omega\in\mathfrak{M}_{m,n}(\mathbb{R})$ is called a \emph{partial permutation} if each row and each column contains at most one entry equal to $1$, with all other entries equal to $0$. If, in addition, $m=n$ and each row and each column contains exactly one entry equal to $1$, it is called a \emph{permutation}.
\end{Def}

Let $\mathfrak{S}_n$ denote the symmetric group consisting of all bijections of the set $[1,n]:=\{1,2,\cdots,n\}$. Then any element $\sigma\in\mathfrak{S}_n$, which is commonly called a permutation, can be identified with a permutation matrix via the correspondence
\begin{align}\label{corres}
\sigma\in\mathfrak{S}_n \leftrightarrow
\begin{pmatrix}
e_{\sigma(1)}\\\vdots\\e_{\sigma(n)}
\end{pmatrix}\in\mathfrak{M}_{n,n}(\mathbb{R}).
\end{align}
Here $\left\{e_i\right\}_{i=1}^n$ is the standard basis of $\mathbb{R}^n$ and each $e_i$ is represented as a row vector with a $1$ in the $i$-th column and $0$ elsewhere. When there is no risk of confusion, we use the same notation for a permutation and the corresponding permutation matrix. 

Following \cite{Ful}, a partial permutation $\omega\in\mathfrak{M}_{m,n}(\mathbb{R})$ can be extended to a permutation matrix by adding appropriate rows and columns. We briefly recall this construction below. Define a sequence $\{a_i\}_{i=1}^m$ by
\begin{align*}
a_i=
\begin{cases}
j &\text{if } \omega_{ij}=1 \text{ for some } j, \\
\infty &\text{otherwise}.
\end{cases}
\end{align*}
Then a permutation $\tilde{\omega}\in\mathfrak{S}_{m+n}$ is defined inductively by setting
\begin{align*}
\tilde{\omega}(i)=
\begin{cases}
a_i &\text{if }i\leq m \text{ and }a_i\leq n,\\
\min\left\{[n+1,m+n]\setminus\left\{\tilde{\omega}(1),\cdots,\tilde{\omega}(i-1)\right\}\right\} &\text{if }i\leq m \text{ and }a_i=\infty,\\
\min\left\{[1,m+n]\setminus\left\{\tilde{\omega}(1),\cdots,\tilde{\omega}(i-1)\right\}\right\} &\text{if }i>m.
\end{cases}
\end{align*}
This permutation satisfies
\begin{align}\label{extproperty}
\tilde{\omega}(i)<\tilde{\omega}(i+1)\quad\text{if}\quad i>m\quad\text{and}\quad\tilde{\omega}^{-1}(j)<\tilde{\omega}^{-1}(j+1)\quad\text{if}\quad j>n.
\end{align}
The corresponding permutation matrix $\tilde{\omega}\in\mathfrak{M}_{m+n,m+n}(\mathbb{R})$ contains $\omega$ as its upper-left $m\times n$ submatrix. Therefore $\tilde{\omega}$ can be considered as an extension of $\omega$. Throughout this paper, the \emph{extension} of $\omega$ to a permutation will refer to $\tilde{\omega}$.

\begin{Def}\normalfont
The \emph{Rothe diagram} $\mathcal{D}(\tilde{\omega})$ of the extension $\tilde{\omega}$ is defined by
\begin{align*}
\mathcal{D}(\tilde{\omega}):=\left\{(i,j)\in [1,m+n]\times [1,m+n]\mid\tilde{\omega}(i)>j\text{ and }\tilde{\omega}^{-1}(j)>i\right\}. 
\end{align*}
\end{Def}
It follows from (\ref{extproperty}) that $\mathcal{D}(\tilde{\omega})\subseteq [1,m]\times [1,n]$. Indeed, if $(i,j)\in\mathcal{D}(\tilde{\omega})$ and $i>m$, then (\ref{extproperty}) implies that
\begin{align*}
\tilde{\omega}(i)<\tilde{\omega}\left(\tilde{\omega}^{-1}(j)\right)=j.
\end{align*}
This yields a contradiction since we must have $\tilde{\omega}(i)>j$. Therefore $i\leq m$, and a similar argument shows that $j\leq n$ as well. 

Moreover, one can observe that the pairs in $\mathcal{D}(\tilde{\omega})$ correspond to the positions that remain after removing each $1$ of $\omega$ together with all entries to its right and below. Based on these observations, the Rothe diagram of a partial permutation is defined as follows.
\begin{Def}\normalfont
The \emph{Rothe diagram} $\mathcal{D}(\omega)$ of a partial permutation $\omega\in\mathfrak{M}_{m,n}(\mathbb{R})$ is defined by $\mathcal{D}(\omega):=\mathcal{D}(\tilde{\omega})$. 
\end{Def}
A \emph{connected component} of $\mathcal{D}(\omega)$ refers to a subset that is connected via adjacent rows and columns. 
\begin{Ex}\label{EX1}\normalfont
For the partial permutation $\omega\in\mathfrak{M}_{3,3}(\mathbb{R})$ given by
\begin{align*}
\omega=
\begin{pmatrix}
0 & 1 & 0\\
0 & 0 & 0\\
1 & 0 & 0
\end{pmatrix},
\end{align*}
the extension $\tilde{\omega}\in\mathfrak{M}_{6,6}(\mathbb{R})$ is
\begin{align*}
\tilde{\omega}=
\begin{pmatrix}
0 & 1 & 0 & 0 & 0 & 0\\
0 & 0 & 0 & 1 & 0 & 0\\
1 & 0 & 0 & 0 & 0 & 0\\
0 & 0 & 1 & 0 & 0 & 0\\
0 & 0 & 0 & 0 & 1 & 0\\
0 & 0 & 0 & 0 & 0 & 1
\end{pmatrix}.
\end{align*}
The Rothe diagram is $\mathcal{D}(\omega)=\{(1,1), (2,1), (2,3)\}$, which has two connected components $\{(1,1),(2,1)\}$ and $\{(2,3)\}$. These positions are precisely those not to the right or below any $1$ of $\omega$.
\end{Ex}
A certain type of partial permutations plays a significant role in the variational approach to determining when the first variation vanishes. 
\begin{Def}\label{Def25}\normalfont
A partial permutation $\omega\in\mathfrak{M}_{m,n}(\mathbb{R})$ is called \emph{vexillary} if its extension $\tilde{\omega}$ is $2143$-pattern avoiding. That is, there exist no $i_1<i_2<i_3<i_4$ such that $\tilde{\omega}(i_2)<\tilde{\omega}(i_1)<\tilde{\omega}(i_4)<\tilde{\omega}(i_3)$. 
\end{Def}
Equivalently, one may use the following characterization. For each $(i,j)\in\mathcal{D}(\omega)$, define
\begin{align*}
\mathcal{R}(i,j):=\left\{r\in [1,m]\mid \omega_{rc}=1\text{ for some }c\in [1,n]\right\}
\end{align*}
and
\begin{align*}
\mathcal{C}(i,j):=\left\{c\in [1,n]\mid \omega_{rc}=1\text{ for some }r\in [1,m]\right\}.
\end{align*}
In other words, $\mathcal{R}(i,j)$ consists of all rows above the $i$-th row that contain a $1$, while $\mathcal{C}(i,j)$ consists of all columns to the left of the $j$-th column that contain a $1$. One can observe that if $(i,j)$ and $(k,l)$ belong to the same connected component of $\mathcal{D}(\omega)$, then $\mathcal{R}(i,j)=\mathcal{R}(k,l)$ and $\mathcal{C}(i,j)=\mathcal{C}(k,l)$. 

Let us denote the restriction of $\omega$ to $\mathcal{R}(i,j)$ and $\mathcal{C}(i,j)$ by $\omega\big|_{\mathcal{R}(i,j),\mathcal{C}(i,j)}$. We now state the following equivalent characterization.
\begin{lemma}\label{vex}
A partial permutation $\omega$ is vexillary if and only if $\omega\big|_{\mathcal{R}(i,j),\mathcal{C}(i,j)}$ is the identity matrix for every $(i,j)\in\mathcal{D}(\omega)$.
\end{lemma}
\begin{proof}
Suppose that $\tilde{\omega}$ contains a $2143$-pattern. That is, there exist $i_1<i_2<i_3<i_4$ such that $\tilde{\omega}(i_2)<\tilde{\omega}(i_1)<\tilde{\omega}(i_4)<\tilde{\omega}(i_3)$. We have
\begin{align*}
\tilde{\omega}(i_3)>\tilde{\omega}(i_4)\quad\text{and}\quad\tilde{\omega}^{-1}\left(\tilde{\omega}(i_4)\right)>i_3,
\end{align*}
which imply that $\left(i_3,\tilde{\omega}(i_4)\right)\in\mathcal{D}(\omega)$. Observe that
\begin{align*}
\{i_1,i_2\}\subseteq\mathcal{R}\left(i_3,\tilde{\omega}(i_4)\right)\quad\text{and}\quad\left\{\tilde{\omega}(i_1),\tilde{\omega}(i_2)\right\}\subseteq\mathcal{C}\left(i_3,\tilde{\omega}(i_4)\right).
\end{align*}
Since $i_1<i_2$ and $\tilde{\omega}(i_1)>\tilde{\omega}(i_2)$, the restriction $\omega|_{\mathcal{R}\left(i_3,\tilde{\omega}(i_4)\right),\mathcal{C}\left(i_3,\tilde{\omega}(i_4)\right)}$ is not the identity matrix.

Now assume that $\omega\big|_{\mathcal{R}(i,j),\mathcal{C}(i,j)}$ is not the identity matrix for some $(i,j)\in\mathcal{D}(\omega)$. Then there exist $k<l$ in $\mathcal{R}(i,j)$ and $t>s$ in $\mathcal{C}(i,j)$ such that $\omega_{kt}=\omega_{ls}=1$. Since $t=\tilde{\omega}(k)$ and $s=\tilde{\omega}(l)$, we obtain $\tilde{\omega}(l)<\tilde{\omega}(k)<j$. On the other hand, from the definition of $\mathcal{D}(\omega)$ we have
\begin{align*}
\tilde{\omega}(i)>j\quad\text{and}\quad \tilde{\omega}^{-1}(j)>i.
\end{align*}
Therefore, 
\begin{align*}
k<l<i<\tilde{\omega}^{-1}(j)
\end{align*}
and
\begin{align*}
\tilde{\omega}(l)<\tilde{\omega}(k)<\tilde{\omega}\left(\tilde{\omega}^{-1}(j)\right)=j<\tilde{\omega}(i).
\end{align*}
Hence $\tilde{\omega}$ contains a $2143$-pattern, thereby completing the proof.
\end{proof}

\subsection{Matrix Schubert varieties}\label{MSchubert}
We recall the definition of real matrix Schubert varieties introduced in \cite{Ful}. Let $\omega\in\mathfrak{M}_{m,n}(\mathbb{R})$ be a partial permutation. The \emph{real matrix Schubert variety} $\overline{X}_{\omega}$ is defined by
\begin{align*}
\overline{X}_{\omega}:=\left\{A\in\mathfrak{M}_{m,n}(\mathbb{R})\mid \text{rk}\left(A_{[p,q]}\right)\leq\text{rk}\left(\omega_{[p,q]}\right)\text{ for all }1\leq p\leq m, 1\leq q\leq n\right\}.
\end{align*}
Here $A_{[p,q]}$ denotes the upper-left $p\times q$ submatrix of $A$ and $\text{rk}$ denotes the rank. It is known that $\overline{X}_{\omega}$ is an irreducible algebraic variety of codimension equal to the size of the Rothe diagram $\left|\mathcal{D}(\omega)\right|$. The defining ideal of $\overline{X}_{\omega}$ is generated by all minors of size $1+\text{rk}\left(\omega_{[p,q]}\right)$ of $\mathbb{X}_{[p,q]}$ for each $(p,q)\in\mathcal{D}(\omega)$, where $\mathbb{X}=(x_{ij})$ is the $m\times n$ matrix of variables. 

\begin{Ex}[The determinantal varieties]\label{ex27}\normalfont
Given $m$, $n$, and $0<r<\min \{m,n\}$, the set of all $m\times n$ matrices of rank at most $r$ is called the determinantal variety. It can be viewed as the real matrix Schubert variety $\overline{X}_{\omega}$ corresponding to
\begin{align*}
\omega=
\begin{pmatrix}
I_r & O_{r\times(n-r)}\\
O_{(m-r)\times r} & O_{(m-r)\times (n-r)}
\end{pmatrix}\in\mathfrak{M}_{m,n}(\mathbb{R}),
\end{align*}
where $I_r$ and $O_{a\times b}$ denote the $r\times r$ identity and the $a\times b$ zero matrices, respectively.
\end{Ex}

\begin{Ex}\label{ex24}\normalfont
Let $\omega$ be the partial permutation in Example \ref{EX1}. Since $\mathcal{D}(\omega)=\{(1,1), (2,1), (2,3)\}$, the codimension of $\overline{X}_{\omega}$ is $|\mathcal{D}(\omega)|=3$. The defining ideal of $\overline{X}_{\omega}$ is generated by $x_{11}$, $x_{21}$, and all $2\times 2$ minors of 
\begin{align*}
\mathbb{X}_{[2,3]}=
\begin{pmatrix}
x_{11} & x_{12} & x_{13}\\
x_{21} & x_{22} & x_{23}
\end{pmatrix}.
\end{align*}
Thus $\overline{X}_{\omega}$ is the set of all $3\times 3$ matrices $A=(A_{ij})$ satisfying
\begin{align*}
A_{11}=A_{21}=0\quad\text{and}\quad A_{12}A_{23}-A_{22}A_{13}=0.
\end{align*}
Hence $\overline{X}_{\omega}$ can be identified with $Cl_3\times\mathbb{R}^3\subset\mathbb{R}^7\subset\mathbb{R}^9$, where $Cl_3\subset\mathbb{R}^4$ is the cone over the Clifford torus. This implies that $\overline{X}_{\omega}$ is a minimal cone in $\mathfrak{M}_{3,3}(\mathbb{R})$. 
\end{Ex}
As illustrated in Example \ref{ex24}, $\overline{X}_{\omega}$ may contain singularities in general. However, every $\overline{X}_{\omega}$ contains an open dense subset consisting of regular points, defined by
\begin{align*}
X_{\omega}:=\left\{A\in\mathfrak{M}_{m,n}(\mathbb{R})\mid \text{rk}\left(A_{[p,q]}\right)=\text{rk}\left(\omega_{[p,q]}\right)\text{ for all }1\leq p\leq m, 1\leq q\leq n\right\}.
\end{align*}
For $A, B\in\mathfrak{M}_{m,n}(\mathbb{R})$, one has
\begin{align}\label{orbit}
\text{rk}\left(A_{[p,q]}\right)=\text{rk}\left(B_{[p,q]}\right)&\text{ for all }1\leq p\leq m, 1\leq q\leq n\nonumber\\
&\Leftrightarrow B=\lambda A\mu\text{ for some }\lambda\in B^{-}_m, \mu\in B^{+}_n,
\end{align}
where $B^{-}_m$ denotes the group of all invertible lower triangular $m\times m$ matrices and $B^{+}_n$ denotes the group of all invertible upper triangular $n\times n$ matrices. It follows that $X_{\omega}$ is the orbit of $\omega$ under the left group action of $B^{-}_m\times B^{+}_n$, given by $(b_{-},b_{+})\cdot \omega=b_{-}\omega b^{-1}_{+}$. 

We are mainly concerned with the minimality of $X_{\omega}$ as a smooth submanifold of $\mathfrak{M}_{m,n}(\mathbb{R})$. For this purpose, we provide an explicit construction of a normal frame of $X_{\omega}$ in the next subsection.

\subsection{A normal frame of $X_{\omega}$}\label{normalframe}
For each $(\alpha,\beta)\in\mathcal{D}(\omega)$, we define
\begin{align*}
f_{(\alpha,\beta)}:=\det \left(\mathbb{X}\big|_{\mathcal{R}(\alpha,\beta)\cup\{\alpha\},\mathcal{C}(\alpha,\beta)\cup\{\beta\}}\right)\in\mathbb{R}[x_{11},\cdots,x_{mn}],
\end{align*}
where $\mathbb{X}=(x_{ij})$ is the $m\times n$ matrix of variables, and $\mathbb{X}\big|_{\mathcal{R}(\alpha,\beta)\cup\{\alpha\},\mathcal{C}(\alpha,\beta)\cup\{\beta\}}$ denotes the submatrix of $\mathbb{X}$ obtained by restricting to the rows in $\mathcal{R}(\alpha,\beta)\cup\{\alpha\}$ and the columns in $\mathcal{C}(\alpha,\beta)\cup\{\beta\}$, respectively. As mentioned at the beginning of Section \ref{MSchubert}, $f_{(\alpha,\beta)}$ belongs to the defining ideal of $\overline{X}_{\omega}$ for every $(\alpha,\beta)\in\mathcal{D}(\omega)$. Thus, 
\begin{align*}
\overline{\nabla}f_{(\alpha,\beta)}(A):=\left(\frac{\partial f_{(\alpha,\beta)}}{\partial x_{ij}}(A)\right)\in\mathfrak{M}_{m,n}(\mathbb{R})
\end{align*}
is a normal vector to $X_{\omega}$ at each point $A\in X_{\omega}$. The following results show that these normal vectors indeed form a normal frame of $X_{\omega}$.

We first establish the following lemma.
\begin{lemma}\label{Lem25}
For every $(\alpha,\beta)\in\mathcal{D}(\omega)$ and $A\in X_{\omega}$, we have
\begin{align*}
\det \left(A\big|_{\mathcal{R}(\alpha,\beta),\mathcal{C}(\alpha,\beta)}\right)\neq 0
\end{align*}
provided that $\textup{rk}\left(\omega_{[\alpha,\beta]}\right)\geq1$.
\end{lemma}
\begin{proof}
Suppose $(\alpha,\beta)\in\mathcal{D}(\omega)$ with $k:=\text{rk}\left(\omega_{[\alpha,\beta]}\right)\geq1$. Write $\mathcal{R}(\alpha,\beta)=\{r_1,\cdots,r_k\}$. Then $\mathcal{C}(\alpha,\beta)=\{\tilde{\omega}(r_1),\cdots,\tilde{\omega}(r_k)\}$, where $\tilde{\omega}$ is the extension of $\omega$ to a permutation. 

For $A\in X_{\omega}$, consider $A'\in\mathfrak{M}_{k,k}(\mathbb{R})$ with entries given by $A'_{st}=A_{r_s\tilde{\omega}(r_t)}$. Then 
\begin{align*}
\left|\det\left(A\big|_{\mathcal{R}(\alpha,\beta),\mathcal{C}(\alpha,\beta)}\right)\right|=\left|\det\left(A'\right)\right|.
\end{align*}
We may write $A=\lambda \omega \mu$ for some $\lambda\in B^{-}_m$ and $\mu\in B^{+}_n$ as in (\ref{orbit}). This gives
\begin{align*}
A'_{st}=\left(\lambda \omega \mu\right)_{r_s\tilde{\omega}(r_t)}=\sum_{i,j}\lambda_{r_si}\omega_{ij}\mu_{j\tilde{\omega}(r_t)}.
\end{align*}
Since $\lambda$ is lower triangular, $\lambda_{r_si}=0$ whenever $i\geq\alpha$. For $i<\alpha$, there are two cases: either $\tilde{\omega}(i)\geq\beta$, in which case $\omega_{ij}=0$ for all $j$, or else $i=r_h$ for some $h$, in which case $\omega_{r_hj}=1$ for $j=\tilde{\omega}(r_h)$ and all other entries vanish. Thus,
\begin{align*}
A'_{st}=\sum_{1\leq h\leq k}\lambda_{r_sr_h}\mu_{\tilde{\omega}(r_h)\tilde{\omega}(r_t)}.
\end{align*}
This implies that 
\begin{align*}
\left|\det\left(A'\right)\right|=\left|\det\left(\lambda\big|_{\mathcal{R}(\alpha,\beta),\mathcal{R}(\alpha,\beta)}\right)\cdot\det\left(\mu\big|_{\mathcal{C}(\alpha,\beta),\mathcal{C}(\alpha,\beta)}\right)\right|.
\end{align*}
Both determinants are nonzero since $\lambda$ and $\mu$ are both invertible triangular matrices. Hence $\det\left(A'\right)\neq0$, which proves the lemma.
\end{proof}

With this lemma at hand, we can prove the following proposition. 
\begin{prop}\label{Prop26}
At each point $A\in X_{\omega}$, the set
\begin{align*}
\left\{\overline{\nabla}f_{(\alpha,\beta)}(A)\mid (\alpha,\beta)\in\mathcal{D}(\omega)\right\}
\end{align*}
forms a basis of the normal space to $X_{\omega}$.
\end{prop}
\begin{proof}
Suppose the Rothe diagram $\mathcal{D}(\omega)$ is ordered lexicographically. Let $P$ be the $|\mathcal{D}(\omega)|\times|\mathcal{D}(\omega)|$ matrix with rows and columns indexed by $\mathcal{D}(\omega)$, whose entry in row $(\alpha,\beta)$ and column $(\epsilon,\delta)$ is 
\begin{align*}
P_{(\alpha,\beta)(\epsilon,\delta)}=\frac{\partial f_{(\alpha,\beta)}}{\partial x_{\epsilon\delta}}(A).
\end{align*}

Consider $(\alpha,\beta), (\epsilon,\delta)\in\mathcal{D}(\omega)$ with $(\alpha,\beta)<(\epsilon,\delta)$; that is, either $\alpha<\epsilon$ or $\alpha=\epsilon$ and $\beta<\delta$. In either case, $\epsilon\notin\mathcal{R}(\alpha,\beta)\cup\{\alpha\}$ or $\delta\notin\mathcal{C}(\alpha,\beta)\cup\{\beta\}$. Hence the variable $x_{\epsilon\delta}$ does not appear in $f_{(\alpha,\beta)}$, and therefore
\begin{align*}
\frac{\partial f_{(\alpha,\beta)}}{\partial x_{\epsilon\delta}}\equiv0.
\end{align*}
It follows that $P$ is lower triangular.

Moreover, the diagonal entries of $P$ are nonzero. Indeed, if $\text{rk}\left(\omega_{[\alpha,\beta]}\right)=0$, then $f_{(\alpha,\beta)}=x_{\alpha\beta}$ and hence
\begin{align*}
P_{(\alpha,\beta)(\alpha,\beta)}=\frac{\partial f_{(\alpha,\beta)}}{\partial x_{\alpha\beta}}(A)=1\neq0.
\end{align*}
If $\text{rk}\left(\omega_{[\alpha,\beta]}\right)\geq 1$, then by Lemma \ref{Lem25} we have
\begin{align*}
P_{(\alpha,\beta)(\alpha,\beta)}=\frac{\partial f_{(\alpha,\beta)}}{\partial x_{\alpha\beta}}(A)=\det \left(A\big|_{\mathcal{R}(\alpha,\beta),\mathcal{C}(\alpha,\beta)}\right)\neq 0.
\end{align*}

Therefore $P$ is lower triangular with nonzero diagonal entries, hence invertible. Consequently, the set of normal vectors
\begin{align*}
\left\{\overline{\nabla}f_{(\alpha,\beta)}(A)\mid (\alpha,\beta)\in\mathcal{D}(\omega)\right\}
\end{align*}
is linearly independent and generates a space of dimension $|\mathcal{D}(\omega)|$. Since the codimension of $X_{\omega}$ is also known to be $|\mathcal{D}(\omega)|$, the proposition follows.
\end{proof}
It follows from Proposition \ref{Prop26} that the mean curvature vector at a point $A\in X_{\omega}$ vanishes if and only if the Hessian of $f_{(\alpha,\beta)}$ has vanishing trace in the tangent space, for all $(\alpha,\beta)\in\mathcal{D}(\omega)$. As discussed in \cite{BCH}, since each variable appears in $f_{(\alpha,\beta)}$ at most once, the Hessian of $f_{(\alpha,\beta)}$ has vanishing trace in the ambient space. Hence the above condition is equivalent to the vanishing of the trace in the normal space. For later use, we record this in the following lemma.
\begin{lemma}\label{minimalcheck}
For each point $A\in X_{\omega}$, the mean curvature vector of $X_{\omega}$ at $A$ vanishes if and only if the Hessian of $f_{(\alpha,\beta)}$ has vanishing trace in the normal space of $X_{\omega}$ at $A$, for all $(\alpha,\beta)\in\mathcal{D}(\omega)$.
\end{lemma}


\section{Computations on minors}\label{lemmaminor}
The computation of the trace of the Hessian of $f_{(\alpha,\beta)}$ relies on the evaluation of its first and second derivatives, which in turn reduce to computations of minors of certain matrices. We establish these computations in a simplified form.

Let $\sigma\in\mathfrak{M}_{n,n}(\mathbb{R})$ be a permutation not equal to the identity matrix. Under the correspondence (\ref{corres}), it corresponds to an element of $\mathfrak{S}_n\setminus\{id\}$, which we also denote by $\sigma$. Since $\sigma\neq id$, there exists $i\in[1,n]$ with $\sigma(i)>\sigma(i+1)$. Let
\begin{align*}
l_\sigma:=\min\left\{i\in [1,n]\mid \sigma(i)>\sigma(i+1)\right\}.
\end{align*}
For nonzero real numbers $y_1$, $y_2$, $y_3$, define $\widehat{\sigma}(y_1,y_2,y_3)\in\mathfrak{M}_{n+1,n+1}(\mathbb{R})$ by
\begin{align*}
\widehat{\sigma}(y_1,y_2,y_3)=
\begin{pmatrix}
\sigma & 0\\
0 & 0
\end{pmatrix}
+y_1E_{l_\sigma,n+1}+y_2E_{n+1,\sigma(l_\sigma)}+y_1y_2E_{n+1,n+1}+y_3E_{l_\sigma+1,\sigma(l_\sigma)},
\end{align*}
where $E_{i,j}$ denotes the matrix with a $1$ in the $(i,j)$ entry and zeros elsewhere. Hence, $\widehat{\sigma}(y_1,y_2,y_3)$ has the following form:
\begin{align}\label{form}
\begin{pmatrix}
e_{\sigma(1)}\\
\vdots\\
e_{\sigma(l_\sigma-1)}\\
e_{\sigma(l_\sigma)}+y_1e_{n+1}\\
e_{\sigma(l_\sigma+1)}+y_3e_{\sigma(l_\sigma)}\\
e_{\sigma(l_\sigma+2)}\\
\vdots\\
e_{\sigma(n)}\\
y_2e_{\sigma(l_\sigma)}+y_1y_2e_{n+1}
\end{pmatrix}
=
\begin{blockarray}{cc*{6}{c}}
&&&\sigma(l_\sigma+1)&&\sigma(l_\sigma)&&n+1\\
&&&\downarrow&&\downarrow&&\downarrow\\
\begin{block}{cc[*{6}{c}]}
&&1&&&&&\\
&&&\ddots&&&&\\
&&&&1&&&\\
l_\sigma&\rightarrow&&&&1&&y_1\\
l_\sigma+1&\rightarrow&&1&&y_3&&\\
&&&&&&\ddots&\\
n+1&\rightarrow&&&&y_2&&y_1y_2\\
\end{block}
\end{blockarray}
\end{align}
We first compute the $(i,j)$-cofactor of $\widehat{\sigma}(y_1,y_2,y_3)$.

\begin{lemma}\label{Lem32}
We have
\begin{align*}
\left.\frac{\textup{d}}{\textup{d} t}\right|_{t=0}\det\left(\widehat{\sigma}(y_1,y_2,y_3)+tE_{i,j}\right)=
\begin{cases}
\det\sigma &\text{if }(i,j)=(n+1,n+1),\\
-y_1\det\sigma &\text{if }(i,j)=(n+1,\sigma(l_\sigma)),\\
y_1y_3\det\sigma &\text{if }(i,j)=(n+1,\sigma(l_\sigma+1)),\\
-y_2\det\sigma &\text{if }(i,j)=(l_\sigma,n+1),\\
y_1y_2\det\sigma &\text{if }(i,j)=(l_\sigma,\sigma(l_\sigma)),\\
-y_1y_2y_3\det\sigma &\text{if }(i,j)=(l_\sigma,\sigma(l_\sigma+1)),\\
0 &\text{otherwise}.
\end{cases}
\end{align*}
\end{lemma}
\begin{proof}
We may use the fact that the derivative is equal to the determinant of the matrix obtained by replacing the $i$-th row of $\widehat{\sigma}(y_1,y_2,y_3)$ with $e_j$. Observe that the $l_\sigma$-th row and the $(n+1)$-st row of $\widehat{\sigma}(y_1,y_2,y_3)$ are linearly dependent. If $i\notin\{ l_\sigma, n+1\}$, these two rows remain linearly dependent after the replacement, hence the determinant vanishes. Thus,
\begin{align*}
\left.\frac{\textup{d}}{\textup{d} t}\right|_{t=0}\det\left(\widehat{\sigma}(y_1,y_2,y_3)+tE_{i,j}\right)=0
\end{align*}
for all $i\in[1,n]\setminus\{l_\sigma\}$ and $j\in[1,n+1]$.

The matrices obtained by replacing the $(n+1)$-st row of $\widehat{\sigma}(y_1,y_2,y_3)$ with $e_j$ and by replacing the $l_\sigma$-th row with $e_j$ are given respectively by
\begin{align*}
M(j):=
\begin{pmatrix}
e_{\sigma(1)}\\
\vdots\\
e_{\sigma(l_\sigma-1)}\\
e_{\sigma(l_\sigma)}+y_1e_{n+1}\\
e_{\sigma(l_\sigma+1)}+y_3e_{\sigma(l_\sigma)}\\
e_{\sigma(l_\sigma+2)}\\
\vdots\\
e_{\sigma(n)}\\
e_j
\end{pmatrix},\quad
N(j):=
\begin{pmatrix}
e_{\sigma(1)}\\
\vdots\\
e_{\sigma(l_\sigma-1)}\\
e_j\\
e_{\sigma(l_\sigma+1)}+y_3e_{\sigma(l_\sigma)}\\
e_{\sigma(l_\sigma+2)}\\
\vdots\\
e_{\sigma(n)}\\
y_2e_{\sigma(l_\sigma)}+y_1y_2e_{n+1}
\end{pmatrix}.
\end{align*}

Let $i=n+1$. Then the derivative coincides with $\det M(j)$. If $j=n+1$, by applying elementary row operations to $M(n+1)$ we obtain $\det M(n+1)=\det \sigma$, and therefore
\begin{align*}
\left.\frac{\textup{d}}{\textup{d} t}\right|_{t=0}\det\left(\widehat{\sigma}(y_1,y_2,y_3)+tE_{n+1,n+1}\right)=\det\sigma.
\end{align*}
For $j\leq n$, we may write $j=\sigma(k)$ for some $k\in[1,n]$. If $k\notin\{l_\sigma,l_\sigma+1\}$, then $M(\sigma(k))$ contains two linearly dependent rows, and thus $\det M(\sigma(k))=0$. If $k\in\{l_\sigma,l_\sigma+1\}$, then by applying row operations to $M(\sigma(k))$ we obtain
\begin{align*}
\det M(\sigma(l_\sigma))=-y_1\det\sigma,\quad\det M(\sigma(l_\sigma+1))=y_1y_3\det\sigma.
\end{align*}
Therefore,
\begin{align*}
\left.\frac{\textup{d}}{\textup{d} t}\right|_{t=0}\det\left(\widehat{\sigma}(y_1,y_2,y_3)+tE_{n+1,j}\right)=
\begin{cases}
0 &\text{if }j\in[1,n]\setminus\{\sigma(l_\sigma),\sigma(l_\sigma+1)\},\\
\det\sigma &\text{if }j=n+1,\\
-y_1\det\sigma &\text{if }j=\sigma(l_\sigma),\\
y_1y_3\det\sigma &\text{if }j=\sigma(l_\sigma+1).
\end{cases}
\end{align*}

Now consider the case $i=l_\sigma$. In this case, the derivative is given by $\det N(j)$. If $j=n+1$, a similar computation yields
\begin{align*}
\left.\frac{\textup{d}}{\textup{d} t}\right|_{t=0}\det\left(\widehat{\sigma}(y_1,y_2,y_3)+tE_{l_\sigma,n+1}\right)=\det N(n+1)=-y_2\det\sigma.
\end{align*}
For $j\leq n$, write $j=\sigma(k)$. As in the previous case, we have $\det N(\sigma(k))=0$ for $k\notin\{l_\sigma,l_\sigma+1\}$, while
\begin{align*}
\det N(\sigma(l_\sigma))=y_1y_2\det\sigma,\quad\det N(\sigma(l_\sigma+1))=-y_1y_2y_3\det\sigma.
\end{align*}
Hence,
\begin{align*}
\left.\frac{\textup{d}}{\textup{d} t}\right|_{t=0}\det\left(\widehat{\sigma}(y_1,y_2,y_3)+tE_{l_\sigma,j}\right)=
\begin{cases}
0 &\text{if }j\in[1,n]\setminus\{\sigma(l_\sigma),\sigma(l_\sigma+1)\},\\
-y_2\det\sigma &\text{if }j=n+1,\\
y_1y_2\det\sigma &\text{if }j=\sigma(l_\sigma),\\
-y_1y_2y_3\det\sigma &\text{if }j=\sigma(l_\sigma+1).
\end{cases}
\end{align*}
\end{proof}

We next consider the $(n-1)\times (n-1)$ minors of $\widehat{\sigma}(y_1,y_2,y_3)$.
\begin{lemma}\label{Lem33}
For $I=[1,n+1]\setminus\{i_1,i_2\}$ and $J=[1,n+1]\setminus\{j_1,j_2\}$,
\begin{align*}
\det\left(\widehat{\sigma}(y_1,y_2,y_3)\big|_{I,J}\right)\neq0
\end{align*}
if and only if the sets $\{i_1,i_2\}$ and $\{j_1,j_2\}$ take one of the following forms:
\begin{itemize}
\item[(1)] $\{i_1,i_2\}=\{l_\sigma, n+1\}$;\quad$\{j_1,j_2\}=\{\sigma(l_\sigma),n+1\}$ or $\{\sigma(l_\sigma+1),n+1\}$.
\item[(2)] $\{i_1,i_2\}=\{l_\sigma, l_\sigma+1\}$;\quad$\{j_1,j_2\}=\{\sigma(l_\sigma),\sigma(l_\sigma+1)\}$ or $\{\sigma(l_\sigma+1),n+1\}$.
\item[(3)] For each $i\in[1,n]\setminus\{l_\sigma,l_\sigma+1\}$,\\
$\{i_1,i_2\}=\{l_\sigma, i\}$;\quad$\{j_1,j_2\}=\{\sigma(l_\sigma),\sigma(i)\}$ or $\{\sigma(l_\sigma+1),\sigma(i)\}$ or $\{n+1,\sigma(i)\}$.
\item[(4)] $\{i_1,i_2\}=\{n+1, l_\sigma+1\}$;\quad$\{j_1,j_2\}=\{\sigma(l_\sigma),\sigma(l_\sigma+1)\}$ or $\{\sigma(l_\sigma+1),n+1\}$.
\item[(5)] For each $i\in[1,n]\setminus\{l_\sigma,l_\sigma+1\}$,\\
$\{i_1,i_2\}=\{n+1, i\}$;\quad$\{j_1,j_2\}=\{\sigma(l_\sigma),\sigma(i)\}$ or $\{\sigma(l_\sigma+1),\sigma(i)\}$ or $\{n+1,\sigma(i)\}$.
\end{itemize}
\end{lemma}
\begin{proof}
For the following arguments, it will be useful to recall the form of $\widehat{\sigma}(y_1,y_2,y_3)$ given in (\ref{form}). Since the $l_\sigma$-th row and the $(n+1)$-st row are linearly dependent, any corresponding $(n-1)\times (n-1)$ minor vanishes whenever $\{i_1,i_2\}\cap\{l_\sigma,n+1\}=\emptyset$.
The remaining cases can be divided into the following five types:
\begin{itemize}
\item[(1)] $\{i_1,i_2\}=\{l_\sigma,n+1\}$.
\item[(2)] $\{i_1,i_2\}=\{l_\sigma,l_\sigma+1\}$.
\item[(3)] $\{i_1,i_2\}=\{l_\sigma, i\}$, where $i\in[1,n]\setminus\{l_\sigma,l_\sigma+1\}$.
\item[(4)] $\{i_1,i_2\}=\{n+1, l_\sigma+1\}$.
\item[(5)] $\{i_1,i_2\}=\{n+1, i\}$, where $i\in[1,n]\setminus\{l_\sigma,l_\sigma+1\}$.
\end{itemize}

Consider case $(1)$. After removing the two rows from $\widehat{\sigma}(y_1,y_2,y_3)$, the $(n+1)$-st column becomes a zero vector. Moreover, the remaining vectors in columns $\sigma(l_\sigma)$ and $\sigma(l_\sigma+1)$ are linearly dependent. Hence, for the minor to be nonzero, $\{j_1,j_2\}$ must be either $\{\sigma(l_\sigma),n+1\}$ or $\{\sigma(l_\sigma+1),n+1\}$. In these cases, we obtain
\begin{align*}
\left|\det\left(\widehat{\sigma}(y_1,y_2,y_3)\big|_{I,J}\right)\right|=
\begin{cases}
1 &\text{if }\{j_1,j_2\}=\{\sigma(l_\sigma),n+1\},\\
|y_3| &\text{if }\{j_1,j_2\}=\{\sigma(l_\sigma+1),n+1\}.
\end{cases}
\end{align*}

In case $(2)$, a zero vector appears in column $\sigma(l_\sigma+1)$ after removing the two rows, while the remaining vectors in columns $\sigma(l_\sigma)$ and $n+1$ are linearly dependent. Therefore $\{j_1,j_2\}$ must be either $\{\sigma(l_\sigma),\sigma(l_\sigma+1)\}$ or $\{\sigma(l_\sigma+1),n+1\}$ for the minor to be nonzero, and we have
\begin{align*}
\left|\det\left(\widehat{\sigma}(y_1,y_2,y_3)\big|_{I,J}\right)\right|=
\begin{cases}
|y_1y_2| &\text{if }\{j_1,j_2\}=\{\sigma(l_\sigma),\sigma(l_\sigma+1)\},\\
|y_2| &\text{if }\{j_1,j_2\}=\{\sigma(l_\sigma+1),n+1\}.
\end{cases}
\end{align*}

For case $(3)$, we observe that column $\sigma(i)$ becomes zero after deleting the two rows. Additionally, since the three remaining vectors in columns $\sigma(l_\sigma)$, $\sigma(l_\sigma+1)$, and $n+1$ are linearly dependent, one of them must be included in $\{j_1,j_2\}$ to have a nonzero minor. Hence, the nonzero minors are obtained in the following cases:
\begin{align*}
\left|\det\left(\widehat{\sigma}(y_1,y_2,y_3)\big|_{I,J}\right)\right|=
\begin{cases}
|y_1y_2| &\text{if }\{j_1,j_2\}=\{\sigma(l_\sigma),\sigma(i)\},\\
|y_1y_2y_3| &\text{if }\{j_1,j_2\}=\{\sigma(l_\sigma+1),\sigma(i)\},\\
|y_2| &\text{if }\{j_1,j_2\}=\{n+1,\sigma(i)\}.
\end{cases}
\end{align*}

Next, in case $(4)$, arguments analogous to those in case $(2)$ yield the following:
\begin{align*}
\left|\det\left(\widehat{\sigma}(y_1,y_2,y_3)\big|_{I,J}\right)\right|=
\begin{cases}
|y_1| &\text{if }\{j_1,j_2\}=\{\sigma(l_\sigma),\sigma(l_\sigma+1)\},\\
1 &\text{if }\{j_1,j_2\}=\{\sigma(l_\sigma+1),n+1\},\\
0 &\text{otherwise}.
\end{cases}
\end{align*}

Lastly, in case $(5)$, we may apply arguments analogous to those in case $(3)$. This gives 
\begin{align*}
\left|\det\left(\widehat{\sigma}(y_1,y_2,y_3)\big|_{I,J}\right)\right|=
\begin{cases}
|y_1| &\text{if }\{j_1,j_2\}=\{\sigma(l_\sigma),\sigma(i)\},\\
|y_1y_3| &\text{if }\{j_1,j_2\}=\{\sigma(l_\sigma+1),\sigma(i)\},\\
1 &\text{if }\{j_1,j_2\}=\{n+1,\sigma(i)\},\\
0 &\text{otherwise}.
\end{cases}
\end{align*}
Therefore the lemma is proved.
\end{proof}


\section{Non-minimality}\label{nonminimal}
\setcounter{equation}{0}
In this section we prove the following theorem, which provides a necessary condition for real matrix Schubert varieties to be minimal.
\begin{thm}\label{Thm41}
Let $\omega\in\mathfrak{M}_{m,n}(\mathbb{R})$ be a partial permutation. If $\omega$ is a non-vexillary partial permutation, then $X_{\omega}\subset\mathfrak{M}_{m,n}(\mathbb{R})$ is not minimal.
\end{thm}

\subsection{Non-minimal example}\label{subsec41}
We begin by observing the smallest non-minimal example. Consider the partial permutation
\begin{align*}
\omega=
\begin{pmatrix}
0 & 1 & 0\\
1 & 0 & 0\\
0 & 0 & 0
\end{pmatrix},
\end{align*}
which is not vexillary by Lemma \ref{vex}. The corresponding real matrix Schubert variety $\overline{X}_{\omega}\subset\mathfrak{M}_{3,3}(\mathbb{R})$ is defined by the equations
\begin{align*}
x_{11}=0,\quad\det
\begin{pmatrix}
x_{11} & x_{12} & x_{13}\\
x_{21} & x_{22} & x_{23}\\
x_{31} & x_{32} & x_{33}
\end{pmatrix}=0.
\end{align*}
For nonzero real numbers $y_1$, $y_2$, $y_3$, set
\begin{align*}
\omega(y_1,y_2,y_3):=
\begin{pmatrix}
0 & 1 & y_1\\
1 & y_3 & 0\\
0 & y_2 & y_1y_2
\end{pmatrix}.
\end{align*}
It is straightforward to verify that $\omega(y_1,y_2,y_3)\in X_{\omega}\subset \overline{X}_{\omega}$. We next check the minimality of $X_{\omega}$ along these points using Lemma \ref{minimalcheck}.

Since the Rothe diagram $\mathcal{D}(\omega)$ is $\{(1,1),(3,3)\}$, by Proposition \ref{Prop26} we have a normal frame $\left\{\overline{\nabla}f_{(1,1)}, \overline{\nabla}f_{(3,3)}\right\}$. At $\omega(y_1,y_2,y_3)\in X_{\omega}$, these are 
\begin{align*}
\overline{\nabla}f_{(1,1)}=
\begin{pmatrix}
1 & 0 & 0\\
0 & 0 & 0\\
0 & 0 & 0
\end{pmatrix},\quad
\overline{\nabla}f_{(3,3)}=
\begin{pmatrix}
y_1y_2y_3 & -y_1y_2 & y_2\\
0 & 0 & 0\\
-y_1y_3 & y_1 & -1
\end{pmatrix}.
\end{align*}
Then the Gram matrix is given by
\begin{align*}
G:&=
\begin{pmatrix}
\left\langle\overline{\nabla}f_{(1,1)}, \overline{\nabla}f_{(1,1)}\right\rangle & \left\langle\overline{\nabla}f_{(1,1)}, \overline{\nabla}f_{(3,3)}\right\rangle\\
\left\langle\overline{\nabla}f_{(3,3)}, \overline{\nabla}f_{(1,1)}\right\rangle & \left\langle\overline{\nabla}f_{(3,3)}, \overline{\nabla}f_{(3,3)}\right\rangle
\end{pmatrix}\\&=
\begin{pmatrix}
1 & y_1y_2y_3\\
y_1y_2y_3 & (1+y_2^2)(1+y_1^2+y_1^2y_3^2)
\end{pmatrix},
\end{align*}
and its inverse is 
\begin{align*}
G^{-1}=\frac{1}{1+y_1^2+y_2^2+y_1^2y_2^2+y_1^2y_3^2}
\begin{pmatrix}
(1+y_2^2)(1+y_1^2+y_1^2y_3^2) & -y_1y_2y_3\\
-y_1y_2y_3 & 1
\end{pmatrix}.
\end{align*} 
Furthermore, we have
\begin{align*}
\text{Hess}f_{(3,3)}\left(\overline{\nabla}f_{(1,1)},\overline{\nabla}f_{(1,1)}\right)=\text{Hess}f_{(3,3)}\left(\overline{\nabla}f_{(3,3)},\overline{\nabla}f_{(3,3)}\right)=0,
\end{align*}
while
\begin{align*}
\text{Hess}f_{(3,3)}\left(\overline{\nabla}f_{(1,1)},\overline{\nabla}f_{(3,3)}\right)=-y_3.
\end{align*}
It follows that the trace of the Hessian of $f_{(3,3)}$ in the normal space is 
\begin{align*}
\frac{2y_1y_2y_3^2}{1+y_1^2+y_2^2+y_1^2y_2^2+y_1^2y_3^2},
\end{align*}
and hence nonzero. By Lemma \ref{minimalcheck}, the mean curvature vector of $X_{\omega}$ is therefore nonzero along the points $\omega(y_1,y_2,y_3)$. 

As we shall see, the method used in this example provides the basis for the arguments in the proof of the theorem.

\subsection{Perturbation of $\omega$ in $X_\omega$}\label{subsec42}

Let $\omega\in\mathfrak{M}_{m,n}(\mathbb{R})$ be a non-vexillary partial permutation. Then Lemma \ref{vex} implies that there exists $(\alpha,\beta)\in\mathcal{D}(\omega)$ such that 
\begin{align}\label{star}
\text{rk}\left(\omega_{[\alpha,\beta]}\right)\geq2\quad\text{and}\quad \omega\big|_{\mathcal{R}(\alpha,\beta),\mathcal{C}(\alpha,\beta)}\text{ is not the identity matrix}. 
\end{align}
We take $(\alpha,\beta)\in\mathcal{D}(\omega)$ as one of the pairs that minimize the sum $\alpha+\beta$ so that it is located at the top-left corner within its connected component.

Write $\mathcal{R}(\alpha,\beta)=\{r_1<\cdots<r_k\}$ and $\mathcal{C}(\alpha,\beta)=\{c_1<\cdots<c_k\}$. Under the identification (\ref{corres}), the restriction $\omega\big|_{\mathcal{R}(\alpha,\beta),\mathcal{C}(\alpha,\beta)}$ can be viewed as an element of $\mathfrak{S}_k\setminus\{id\}$, denoted by $\omega^{(\alpha,\beta)}$. As in Section \ref{lemmaminor}, set
\begin{align*}
L:=\min\left\{i\in[1,k]\Bigm| \omega^{(\alpha,\beta)}(i)>\omega^{(\alpha,\beta)}(i+1)\right\}.
\end{align*}
Define $\widehat{\omega}^{(\alpha,\beta)}(y_1,y_2,y_3)\in\mathfrak{M}_{m,n}(\mathbb{R})$ by
\begin{align}\label{omegappp}
\widehat{\omega}^{(\alpha,\beta)}&(y_1,y_2,y_3)\nonumber\\
&=\omega+y_1E_{r_L,\beta}+y_2E_{\alpha,c_{\omega^{(\alpha,\beta)}(L)}}+y_1y_2E_{\alpha,\beta}+y_3E_{r_{L+1},c_{\omega^{(\alpha,\beta)}(L)}},
\end{align}
where $y_1,y_2,y_3$ are nonzero real numbers, and $E_{i,j}$ denotes the matrix with a $1$ in the $(i,j)$ entry and zeros elsewhere. Restricting the rows and columns of $\widehat{\omega}^{(\alpha,\beta)}(y_1,y_2,y_3)$ to $\mathcal{R}(\alpha,\beta)\cup\{\alpha\}$ and $\mathcal{C}(\alpha,\beta)\cup\{\beta\}$, respectively, yields
\begin{align}\label{hathat}
\widehat{\omega}^{(\alpha,\beta)}(y_1,y_2,y_3)\big|_{\mathcal{R}(\alpha,\beta)\cup\{\alpha\},\mathcal{C}(\alpha,\beta)\cup\{\beta\}}=\widehat{\omega^{(\alpha,\beta)}}(y_1,y_2,y_3),
\end{align}
where the last term is defined in Section \ref{lemmaminor}. Since $\widehat{\omega}^{(\alpha,\beta)}(y_1,y_2,y_3)$ is obtained from $\omega$ by applying elementary column operations to the right and row operations to the bottom, it follows that $\widehat{\omega}^{(\alpha,\beta)}(y_1,y_2,y_3)\in X_\omega$. 

In the following subsections, we will show that the trace of the Hessian of $f_{(\alpha,\beta)}$ in the normal space does not vanish at these points. By Lemma \ref{minimalcheck}, this implies that the mean curvature vectors at these points are nonzero, thereby completing the proof.

\subsection{The first and second derivatives of $f_{(i.j)}$}
We consider the first and second derivatives of $f_{(i,j)}$ for $(i,j)\in\mathcal{D}(\omega)$. The next two lemmas are immediate consequences of Lemma \ref{Lem32} and \ref{Lem33}.
\begin{lemma}\label{Lem42}
For simplicity, let us write $\delta=\det \left(\omega^{(\alpha,\beta)}\right)$. Then
\begin{align*}
\frac{\partial f_{(\alpha,\beta)}}{\partial x_{ab}}\left(\widehat{\omega}^{(\alpha,\beta)}(y_1,y_2,y_3)\right)=
\begin{cases}
\delta &\text{if }(a,b)=(\alpha,\beta),\\
-y_1\delta &\text{if }(a,b)=(\alpha,c_{\omega^{(\alpha,\beta)}(L)}),\\
y_1y_3\delta &\text{if }(a,b)=(\alpha,c_{\omega^{(\alpha,\beta)}(L+1)}),\\
-y_2\delta &\text{if }(a,b)=(r_L,\beta),\\
y_1y_2\delta &\text{if }(a,b)=(r_L,c_{\omega^{(\alpha,\beta)}(L)}),\\
-y_1y_2y_3\delta &\text{if }(a,b)=(r_L,c_{\omega^{(\alpha,\beta)}(L+1)}),\\
0 &\text{otherwise}.
\end{cases}
\end{align*}
\end{lemma}
\begin{proof}
Whenever $a\notin \mathcal{R}(\alpha,\beta)\cup\{\alpha\}$ or $b\notin \mathcal{C}(\alpha,\beta)\cup\{\beta\}$, the variable $x_{ab}$ does not appear in $f_{(\alpha,\beta)}$, and therefore
\begin{align*}
\frac{\partial f_{(\alpha,\beta)}}{\partial x_{ab}}=0.
\end{align*}
Now suppose $a\in\mathcal{R}(\alpha,\beta)\cup\{\alpha\}$ and $b\in\mathcal{C}(\alpha,\beta)\cup\{\beta\}$. The lemma then follows from (\ref{hathat}) and Lemma \ref{Lem32} with $\sigma=\omega^{(\alpha,\beta)}$.
\end{proof}

\begin{lemma}\label{Lem43}
The following holds for the second derivatives of $f_{(\alpha,\beta)}$.
\begin{align*}
\frac{\partial^2f_{(\alpha,\beta)}}{\partial x_{a_1b_1}\partial x_{a_2b_2}}\left(\widehat{\omega}^{(\alpha,\beta)}(y_1,y_2,y_3)\right)\neq 0
\end{align*}
if and only if the sets $\{a_1,a_2\}$ and $\{b_1,b_2\}$ take one of the following forms:
\begin{itemize}
\item[(1)] $\{a_1,a_2\}=\{r_L,\alpha\}$;\quad$\{b_1,b_2\}=\{c_{\omega^{(\alpha,\beta)}(L)},\beta\}$ or $\{c_{\omega^{(\alpha,\beta)}(L+1)},\beta\}$.
\item[(2)] $\{a_1,a_2\}=\{r_L,r_{L+1}\}$;\quad$\{b_1,b_2\}=\{c_{\omega^{(\alpha,\beta)}(L+1)},c_{\omega^{(\alpha,\beta)}(L)}\}$ or $\{c_{\omega^{(\alpha,\beta)}(L+1)},\beta\}$.
\item[(3)] For each $r_i\in\mathcal{R}(\alpha,\beta)\setminus\{r_L,r_{L+1}\}$,\\
$\{a_1,a_2\}=\{r_L, r_i\}$;\quad$\{b_1,b_2\}=\{c_{\omega^{(\alpha,\beta)}(L)},c_{\omega^{(\alpha,\beta)}(i)}\}$ or $\{c_{\omega^{(\alpha,\beta)}(L+1)},c_{\omega^{(\alpha,\beta)}(i)}\}$ or $\{\beta, c_{\omega^{(\alpha,\beta)}(i)}\}$.
\item[(4)] $\{a_1,a_2\}=\{\alpha, r_{L+1}\}$;\quad$\{b_1,b_2\}=\{c_{\omega^{(\alpha,\beta)}(L+1)},c_{\omega^{(\alpha,\beta)}(L)}\}$ or $\{c_{\omega^{(\alpha,\beta)}(L+1)},\beta\}$.
\item[(5)] For each $r_i\in\mathcal{R}(\alpha,\beta)\setminus\{r_L,r_{L+1}\}$,\\
$\{a_1,a_2\}=\{\alpha, r_i\}$;\quad$\{b_1,b_2\}=\{c_{\omega^{(\alpha,\beta)}(L)},c_{\omega^{(\alpha,\beta)}(i)}\}$ or $\{c_{\omega^{(\alpha,\beta)}(L+1)},c_{\omega^{(\alpha,\beta)}(i)}\}$ or $\{\beta, c_{\omega^{(\alpha,\beta)}(i)}\}$.
\end{itemize}
\end{lemma}
\begin{proof}
Both variables $x_{a_1b_1}$ and $x_{a_2b_2}$ appear in $f_{(\alpha,\beta)}$ precisely when $\{a_1,a_2\}\subseteq\mathcal{R}(\alpha,\beta)\cup\{\alpha\}$ and $\{b_1,b_2\}\subseteq\mathcal{C}(\alpha,\beta)\cup\{\beta\}$. In this case, Lemma \ref{Lem33} applied to $\sigma=\omega^{(\alpha,\beta)}$, together with (\ref{hathat}), establishes the lemma.
\end{proof}

For $(i,j)\in\mathcal{D}(\omega)\setminus\{(\alpha,\beta)\}$, we consider the derivative
\begin{align*}
\frac{\partial f_{(i,j)}}{\partial x_{ab}}\left(\widehat{\omega}^{(\alpha,\beta)}(y_1,y_2,y_3)\right)
\end{align*}
when $a\leq\alpha$ and $b\leq\beta$. The derivative is trivially zero if $a\notin\mathcal{R}(i,j)\cup\{i\}$ or $b\notin\mathcal{C}(i,j)\cup\{j\}$ as the variable $x_{ab}$ is not contained in $f_{(i,j)}$. Hence we may assume that $a\in\mathcal{R}(i,j)\cup\{i\}$ and $b\in\mathcal{C}(i,j)\cup\{j\}$. In this case, the derivative is, up to sign, the determinant of
\begin{align}\label{restriction}
\widehat{\omega}^{(\alpha,\beta)}(y_1,y_2,y_3)\Big|_{\left(\mathcal{R}(i,j)\cup\{i\}\right)\setminus\{a\},\left(\mathcal{C}(i,j)\cup\{j\}\right)\setminus\{b\}}.
\end{align}
We divide into several cases to determine when the derivative, or equivalently the determinant,  is nonzero.

\begin{lemma}\label{Lem44}
Let $(i,j)\in\mathcal{D}(\omega)\setminus\{(\alpha,\beta)\}$, and let $a\leq \alpha$ and $b\leq \beta$. Assume further that $a\in\mathcal{R}(i,j)\cup\{i\}$ and $b\in\mathcal{C}(i,j)\cup\{j\}$. Then
\begin{align*}
\det\left(\widehat{\omega}^{(\alpha,\beta)}(y_1,y_2,y_3)\Big|_{\left(\mathcal{R}(i,j)\cup\{i\}\right)\setminus\{a\},\left(\mathcal{C}(i,j)\cup\{j\}\right)\setminus\{b\}}\right)=0
\end{align*}
unless $a=i$.
\end{lemma}
\begin{proof}
Suppose $a\neq i$. Then the $i$-th row of $\widehat{\omega}^{(\alpha,\beta)}(y_1,y_2,y_3)$ is contained in the restriction (\ref{restriction}). Observe that all entries in the $i$-th row of $\omega$ are zero in columns $1,\cdots,j$, since $(i,j)\in\mathcal{D}(\omega)$. 

If $i\notin\{r_L,r_{L+1},\alpha\}$, it follows from (\ref{omegappp}) that all entries in the $i$-th row of $\widehat{\omega}^{(\alpha,\beta)}(y_1,y_2,y_3)$ in the columns $\left(\mathcal{C}(i,j)\cup\{j\}\right)\setminus\{b\}$ are also zero, and consequently the determinant is zero.

If $i=r_L$, then since $\omega_{r_Lc_{\omega^{(\alpha,\beta)}(L)}}=1$ and $(i,j)=(r_L,j)\in\mathcal{D}(\omega)$, we must have $j<c_{\omega^{(\alpha,\beta)}(L)}$. This gives $j<\beta$, which implies $\beta\notin\mathcal{C}(i,j)\cup\{j\}$. As a result, the row $i$ of the restriction (\ref{restriction}) is entirely zero. Similarly, if $i=r_{L+1}$, the conditions $\omega_{r_{L+1}c_{\omega^{(\alpha,\beta)}(L+1)}}=1$ and $(i,j)=(r_{L+1},j)\in\mathcal{D}(\omega)$ imply $j<c_{\omega^{(\alpha,\beta)}(L+1)}$. Thus $j<c_{\omega^{(\alpha,\beta)}(L)}$, and again all entries in the row $i$ of the restriction (\ref{restriction}) are zero. In both cases, the determinant vanishes.

If $i=\alpha$, for the row $i$ of the restriction (\ref{restriction}) to contain a nonzero entry, we must have $j\geq c_{\omega^{(\alpha,\beta)}(L)}$. In this case $(i,j)$ also satisfies (\ref{star}), and by assumption we obtain $\alpha+\beta\leq i+j=\alpha+j$. Since $(i,j)=(\alpha,j)\neq(\alpha,\beta)$, this gives $j>\beta\geq b$, so that the $j$-th column of $\widehat{\omega}^{(\alpha,\beta)}(y_1,y_2,y_3)$ is contained in the restriction (\ref{restriction}). However, as $j>\beta$, it follows from (\ref{omegappp}) that the $j$-th column of $\widehat{\omega}^{(\alpha,\beta)}(y_1,y_2,y_3)$ coincides with that of $\omega$. Since $(i,j)=(\alpha,j)\in\mathcal{D}(\omega)$, all entries in the $j$-th column of $\omega$ are zero in rows $1,\cdots,\alpha$. Therefore the restriction (\ref{restriction}) contains a zero column, and the determinant vanishes. Thus, the lemma follows.
\end{proof}
We next consider the case where $a=i$ and $b\neq j$.
\begin{lemma}\label{Lem45}
Let $(i,j)\in\mathcal{D}(\omega)\setminus\{(\alpha,\beta)\}$, and let $i\leq\alpha$ and $b\leq \beta$. If $b\in\mathcal{C}(i,j)$, then
\begin{align*}
\det\left(\widehat{\omega}^{(\alpha,\beta)}(y_1,y_2,y_3)\Big|_{\mathcal{R}(i,j),\left(\mathcal{C}(i,j)\cup\{j\}\right)\setminus\{b\}}\right)\neq 0
\end{align*}
if and only if $r_L<i<r_{L+1}$, $j=\beta$, and $b=c_{\omega^{(\alpha,\beta)}(L)}$.
\end{lemma}
\begin{proof}
Since $b\in\mathcal{C}(i,j)$, the $j$-th column of $\widehat{\omega}^{(\alpha,\beta)}(y_1,y_2,y_3)$ is contained in the restriction (\ref{restriction}). From $(i,j)\in\mathcal{D}(\omega)$, all entries in the $j$-th column of $\omega$ are zero in rows $1,\cdots, i$. Therefore, if $j\notin\{c_{\omega^{(\alpha,\beta)}(L)},\beta\}$, the $j$-th column of $\widehat{\omega}^{(\alpha,\beta)}(y_1,y_2,y_3)$ coincides with that of $\omega$. Then the restriction (\ref{restriction}) contains a zero column, and the determinant vanishes.

If $j=c_{\omega^{(\alpha,\beta)}(L)}$, the conditions $\omega_{r_Lc_{\omega^{(\alpha,\beta)}(L)}}=1$ and $(i,j)=(i,c_{\omega^{(\alpha,\beta)}(L)})\in\mathcal{D}(\omega)$ give $i<r_L$. It follows that $\{r_{L+1},\alpha\}\cap\mathcal{R}(i,j)=\emptyset$, and once again all entries in the column $j$ of the restriction (\ref{restriction}) are zero. 

On the other hand, if $j=\beta$, for the column $j$ of the restriction (\ref{restriction}) to contain a nontrivial entry we must have $i> r_L$. Since $i\leq\alpha$, we deduce $i<\alpha$ as $(i,j)=(i,\beta)\neq(\alpha,\beta)$. If $i\geq r_{L+1}$, then $(i,j)$ satisfies (\ref{star}), contradicting the assumption that $\alpha+\beta$ is minimal among such pairs. Hence $r_L<i<r_{L+1}$. 

If $b\neq c_{\omega^{(\alpha,\beta)}(L)}$, then the restriction (\ref{restriction}) contains two linearly dependent columns, namely $c_{\omega^{(\alpha,\beta)}(L)}$ and $\beta$. Therefore, we must have $b=c_{\omega^{(\alpha,\beta)}(L)}$. In this case, the determinant is, up to sign, $y_1\det\left(\omega^{(\alpha,\beta)}\right)(\neq 0)$. Hence the lemma is proved.
\end{proof}

Finally, if $a=i$ and $b=j$, then since $\widehat{\omega}^{(\alpha,\beta)}(y_1,y_2,y_3)\in X_\omega$, Proposition \ref{Prop26} implies that the derivative 
\begin{align*}
\frac{\partial f_{(i,j)}}{\partial x_{ij}}\left(\widehat{\omega}^{(\alpha,\beta)}(y_1,y_2,y_3)\right)
\end{align*}
does not vanish. Combining these observations, we obtain the following lemma.

\begin{lemma}\label{Lem46}
Let $(i,j)\in\mathcal{D}(\omega)\setminus\{(\alpha,\beta)\}$. For $a\leq \alpha$ and $b\leq \beta$, the derivative
\begin{align*}
\frac{\partial f_{(i,j)}}{\partial x_{ab}}\left(\widehat{\omega}^{(\alpha,\beta)}(y_1,y_2,y_3)\right)
\end{align*}
does not vanish if and only if either
\begin{align*}
(a,b)=(i,j),
\end{align*}
or, if in addition $j=\beta$ and $r_L<i<r_{L+1}$,
\begin{align*}
(a,b)=(i,c_{\omega^{(\alpha,\beta)}(L)}).
\end{align*}
If these additional conditions are not satisfied, then only the first case can occur.
\end{lemma}

\subsection{The Hessian of $f_{(\alpha,\beta)}$}
We evaluate the Hessian of $f_{(\alpha,\beta)}$ with respect to the normal frame obtained in Proposition \ref{Prop26}.
\begin{lemma}\label{Lem47}
For all $(\epsilon_1,\eta_1),(\epsilon_2,\eta_2)\in\mathcal{D}(\omega)\setminus\{(\alpha,\beta)\}$,
\begin{align*}
\textup{Hess}f_{(\alpha,\beta)}\left(\overline{\nabla}f_{(\epsilon_1,\eta_1)},\overline{\nabla}f_{(\epsilon_2,\eta_2)}\right)=0
\end{align*}
at every point $\widehat{\omega}^{(\alpha,\beta)}(y_1,y_2,y_3)\in X_\omega$.
\end{lemma}
\begin{proof}
Expanding the Hessian gives
\begin{align*}
\text{Hess}f_{(\alpha,\beta)}\left(\overline{\nabla}f_{(\epsilon_1,\eta_1)},\overline{\nabla}f_{(\epsilon_2,\eta_2)}\right)=\sum_{\substack{1\leq a_1,a_2\leq m,\\ 1\leq b_1,b_2\leq n}}\frac{\partial f_{(\epsilon_1,\eta_1)}}{\partial x_{a_1b_1}}\cdot \frac{\partial^2f_{(\alpha,\beta)}}{\partial x_{a_1b_1}\partial x_{a_2b_2}}\cdot \frac{\partial f_{(\epsilon_2,\eta_2)}}{\partial x_{a_2b_2}}.
\end{align*}
Observe that all cases $\{a_1,a_2\}$ and $\{b_1,b_2\}$ for which the second derivative of $f_{(\alpha,\beta)}$ at $\widehat{\omega}^{(\alpha,\beta)}(y_1,y_2,y_3)$ is nonzero, as identified in Lemma \ref{Lem43}, satisfy $a_1,a_2\leq\alpha$ and $b_1,b_2\leq\beta$, respectively. Moreover, none of them satisfies the additional conditions in Lemma \ref{Lem46}. Applying Lemma \ref{Lem46} to the first derivatives, the only remaining term in the above expansion is
\begin{align*}
\frac{\partial f_{(\epsilon_1,\eta_1)}}{\partial x_{\epsilon_1\eta_1}}\cdot \frac{\partial^2f_{(\alpha,\beta)}}{\partial x_{\epsilon_1\eta_1}\partial x_{\epsilon_2\eta_2}}\cdot \frac{\partial f_{(\epsilon_2,\eta_2)}}{\partial x_{\epsilon_2\eta_2}}
\end{align*}
provided that $\epsilon_1\neq\epsilon_2$ and $\eta_1\neq\eta_2$. To complete the proof, we show that for $(\epsilon_1,\eta_1),(\epsilon_2,\eta_2)\in\mathcal{D}(\omega)\setminus\{(\alpha,\beta)\}$, the pairs $\{\epsilon_1,\epsilon_2\}$ and $\{\eta_1,\eta_2\}$ are not among the cases listed in Lemma \ref{Lem43}, and hence the above second derivative vanishes.

Consider the rows $r_L,r_{L+1},\alpha$ and the columns $c_{\omega^{(\alpha,\beta)}(L+1)},c_{\omega^{(\alpha,\beta)}(L)},\beta$. Among the pairs formed by these rows and columns, the only one contained in $\mathcal{D}(\omega)\setminus\{(\alpha,\beta)\}$ is $(r_L,c_{\omega^{(\alpha,\beta)}(L+1)})$, since there are $1$'s located at $(r_L,c_{\omega^{(\alpha,\beta)}(L)})$ and $(r_{L+1},c_{\omega^{(\alpha,\beta)}(L+1)})$. It follows from this observation that such $(\epsilon_1,\eta_1),(\epsilon_2,\eta_2)\in\mathcal{D}(\omega)\setminus\{(\alpha,\beta)\}$ cannot occur in cases $(1)$, $(2)$, or $(4)$ of Lemma \ref{Lem43}.

Case $(5)$ of Lemma \ref{Lem43} is excluded as well, since $(\alpha,c_{\omega^{(\alpha,\beta)}(i)})\notin\mathcal{D}(\omega)$ and $(r_i,c_{\omega^{(\alpha,\beta)}(i)})\notin\mathcal{D}(\omega)$. If such $(\epsilon_1,\eta_1),(\epsilon_2,\eta_2)\in\mathcal{D}(\omega)\setminus\{(\alpha,\beta)\}$ were to occur in case $(3)$, then as $(r_i,c_{\omega^{(\alpha,\beta)}(i)})\notin\mathcal{D}(\omega)$ we would need $(r_L, c_{\omega^{(\alpha,\beta)}(i)})\in\mathcal{D}(\omega)$, and this must be included in such a pair. This can only happen when $r_i>r_L$. 

If $r_i>r_L$, then since there are $1$'s located at $(r_L,c_{\omega^{(\alpha,\beta)}(L)})$ and $(r_{L+1},c_{\omega^{(\alpha,\beta)}(L+1)})$, the only remaining possibility is $(r_i,\beta)\in\mathcal{D}(\omega)$ so that the pair would necessarily consist of $(r_L, c_{\omega^{(\alpha,\beta)}(i)})$ and $(r_i,\beta)$. However, in this case $(r_i,\beta)$ satisfies (\ref{star}), which contradicts the assumption that $\alpha+\beta$ is minimal. Hence case $(3)$ is also excluded, thereby completing the proof.
\end{proof}

\begin{lemma}\label{Lem48}
For all $(\epsilon,\eta)\in\mathcal{D}(\omega)\setminus\{(\alpha,\beta)\}$,
\begin{align*}
\textup{Hess}f_{(\alpha,\beta)}\left(\overline{\nabla}f_{(\alpha,\beta)},\overline{\nabla}f_{(\epsilon,\eta)}\right)=
\begin{cases}
(-1)^{L+\omega^{(\alpha,\beta)}(L+1)}y_3 &\text{if }(\epsilon,\eta)=(r_L, c_{\omega^{(\alpha,\beta)}(L+1)}),\\
0 &\text{otherwise},
\end{cases}
\end{align*}
at every point $\widehat{\omega}^{(\alpha,\beta)}(y_1,y_2,y_3)\in X_\omega$.
\end{lemma}
\begin{proof}
By arguments similar to those in the proof of Lemma \ref{Lem47}, the Hessian expands as
\begin{align*}
\text{Hess}f_{(\alpha,\beta)}\left(\overline{\nabla}f_{(\alpha,\beta)},\overline{\nabla}f_{(\epsilon,\eta)}\right)&=\sum_{\substack{1\leq a_1,a_2\leq m,\\ 1\leq b_1,b_2\leq n}}\frac{\partial f_{(\alpha,\beta)}}{\partial x_{a_1b_1}}\cdot \frac{\partial^2f_{(\alpha,\beta)}}{\partial x_{a_1b_1}\partial x_{a_2b_2}}\cdot \frac{\partial f_{(\epsilon,\eta)}}{\partial x_{a_2b_2}}\\
&=\sum_{\substack{1\leq a_1\leq m,\\1\leq b_1\leq n}}\frac{\partial f_{(\alpha,\beta)}}{\partial x_{a_1b_1}}\cdot\frac{\partial^2 f_{(\alpha,\beta)}}{\partial x_{a_1b_1}\partial x_{\epsilon\eta}}\cdot\frac{\partial f_{(\epsilon,\eta)}}{\partial x_{\epsilon\eta}}
\end{align*}
at $\widehat{\omega}^{(\alpha,\beta)}(y_1,y_2,y_3)$. Further application of Lemma \ref{Lem42} gives
\begin{align*}
\text{Hess}f_{(\alpha,\beta)}&\left(\overline{\nabla}f_{(\alpha,\beta)},\overline{\nabla}f_{(\epsilon,\eta)}\right)\\
&=\sum_{\substack{a_1\in\{r_L,\alpha\},\\b_1\in\{c_{\omega^{(\alpha,\beta)}(L+1)},c_{\omega^{(\alpha,\beta)}(L)},\beta\}}}\frac{\partial f_{(\alpha,\beta)}}{\partial x_{a_1b_1}}\cdot\frac{\partial^2 f_{(\alpha,\beta)}}{\partial x_{a_1b_1}\partial x_{\epsilon\eta}}\cdot\frac{\partial f_{(\epsilon,\eta)}}{\partial x_{\epsilon\eta}}.
\end{align*}
We now determine the cases in which the second derivative among the terms in the above summation does not vanish.

Suppose that $\{a_1,\epsilon\}$ and $\{b_1,\eta\}$ correspond to case $(3)$ or $(5)$ of Lemma \ref{Lem43}. Then, apart from the admissible possibilities for $(a_1,b_1)$ in the summation, the only remaining candidate for $(\epsilon,\eta)$ is $(r_i,c_{\omega^{(\alpha,\beta)}(i)})$. Since this position does not belong to $\mathcal{D}(\omega)$, it cannot coincide with $(\epsilon,\eta)\in\mathcal{D}(\omega)\setminus\{(\alpha,\beta)\}$.

As noted in the proof of Lemma \ref{Lem47}, among the pairs formed by the rows $r_L,r_{L+1},\alpha$ and the columns $c_{\omega^{(\alpha,\beta)}(L+1)},c_{\omega^{(\alpha,\beta)}(L)},\beta$, the only one contained in $\mathcal{D}(\omega)\setminus\{(\alpha,\beta)\}$ is $(r_L, c_{\omega^{(\alpha,\beta)}(L+1)})$. Hence, for the second derivative to be nonzero, we must have $(\epsilon,\eta)=(r_L, c_{\omega^{(\alpha,\beta)}(L+1)})$ in cases $(1)$, $(2)$, or $(4)$. This further implies $a_1=\alpha$ since $a_1\in\{r_L,\alpha\}$ and $a_1\neq\epsilon$. 

Among cases $(1)$, $(2)$, and $(4)$ of Lemma \ref{Lem43}, the only one containing $c_{\omega^{(\alpha,\beta)}(L+1)}$ arises from case $(1)$. In this case, we must have $b_1=\beta$ for the second derivative to be nonzero. Consequently, provided that $(\epsilon,\eta)=(r_L, c_{\omega^{(\alpha,\beta)}(L+1)})$, we obtain
\begin{align*}
\text{Hess}f_{(\alpha,\beta)}&\left(\overline{\nabla}f_{(\alpha,\beta)},\overline{\nabla}f_{(r_L, c_{\omega^{(\alpha,\beta)}(L+1)})}\right)\\&=\frac{\partial f_{(\alpha,\beta)}}{\partial x_{\alpha\beta}}\cdot\frac{\partial^2 f_{(\alpha,\beta)}}{\partial x_{\alpha\beta}\partial x_{r_Lc_{\omega^{(\alpha,\beta)}(L+1)}}}\cdot\frac{\partial f_{(r_L, c_{\omega^{(\alpha,\beta)}(L+1)})}}{\partial x_{r_Lc_{\omega^{(\alpha,\beta)}(L+1)}}}\\
&=(-1)^{L+\omega^{(\alpha,\beta)}(L+1)}y_3
\end{align*}
at $\widehat{\omega}^{(\alpha,\beta)}(y_1,y_2,y_3)$. Otherwise, the Hessian vanishes. This completes the proof.
\end{proof}

\begin{lemma}\label{Lem49}
At every point $\widehat{\omega}^{(\alpha,\beta)}(y_1,y_2,y_3)\in X_\omega$,
\begin{align*}
\textup{Hess}f_{(\alpha,\beta)}\left(\overline{\nabla}f_{(\alpha,\beta)},\overline{\nabla}f_{(\alpha,\beta)}\right)=0.
\end{align*}
\end{lemma}
\begin{proof}
By Lemma \ref{Lem42} and Lemma \ref{Lem46}, the only nonzero terms in the expansion of the Hessian at $\widehat{\omega}^{(\alpha,\beta)}(y_1,y_2,y_3)$ arise from case $(1)$ of Lemma \ref{Lem46}. Thus, we obtain
\begin{align*}
\frac{1}{2}\text{Hess}f_{(\alpha,\beta)}\left(\overline{\nabla}f_{(\alpha,\beta)},\overline{\nabla}f_{(\alpha,\beta)}\right)
&=\frac{\partial f_{(\alpha,\beta)}}{\partial x_{r_Lc_{\omega^{(\alpha,\beta)}(L)}}}\cdot\frac{\partial^2 f_{(\alpha,\beta)}}{\partial x_{r_Lc_{\omega^{(\alpha,\beta)}(L)}}\partial x_{\alpha\beta}}\cdot\frac{\partial f_{(\alpha,\beta)}}{\partial x_{\alpha\beta}}\\
&+\frac{\partial f_{(\alpha,\beta)}}{\partial x_{\alpha c_{\omega^{(\alpha,\beta)}(L)}}}\cdot\frac{\partial^2f_{(\alpha,\beta)}}{\partial x_{\alpha c_{\omega^{(\alpha,\beta)}(L)}}\partial x_{r_L\beta}}\cdot\frac{\partial f_{(\alpha,\beta)}}{\partial x_{r_L\beta}}\\
&+\frac{\partial f_{(\alpha,\beta)}}{\partial x_{r_Lc_{\omega^{(\alpha,\beta)}(L+1)}}}\cdot\frac{\partial^2 f_{(\alpha,\beta)}}{\partial x_{r_Lc_{\omega^{(\alpha,\beta)}(L+1)}}\partial x_{\alpha\beta}}\cdot\frac{\partial f_{(\alpha,\beta)}}{\partial x_{\alpha\beta}}\\
&+\frac{\partial f_{(\alpha,\beta)}}{\partial x_{\alpha c_{\omega^{(\alpha,\beta)}(L+1)}}}\cdot\frac{\partial^2 f_{(\alpha,\beta)}}{\partial x_{\alpha c_{\omega^{(\alpha,\beta)}(L+1)}}\partial x_{r_L\beta}}\cdot\frac{\partial f_{(\alpha,\beta)}}{\partial x_{r_L\beta}}.
\end{align*}
Since
\begin{align*}
&\frac{\partial^2 f_{(\alpha,\beta)}}{\partial x_{r_Lc_{\omega^{(\alpha,\beta)}(L)}}\partial x_{\alpha\beta}}=-\frac{\partial^2f_{(\alpha,\beta)}}{\partial x_{\alpha c_{\omega^{(\alpha,\beta)}(L)}}\partial x_{r_L\beta}},\\
&\frac{\partial^2 f_{(\alpha,\beta)}}{\partial x_{r_Lc_{\omega^{(\alpha,\beta)}(L+1)}}\partial x_{\alpha\beta}}=-\frac{\partial^2 f_{(\alpha,\beta)}}{\partial x_{\alpha c_{\omega^{(\alpha,\beta)}(L+1)}}\partial x_{r_L\beta}},
\end{align*}
it follows that
\begin{align*}
\frac{1}{2}&\text{Hess}f_{(\alpha,\beta)}\left(\overline{\nabla}f_{(\alpha,\beta)},\overline{\nabla}f_{(\alpha,\beta)}\right)\\
&=\frac{\partial^2 f_{(\alpha,\beta)}}{\partial x_{r_Lc_{\omega^{(\alpha,\beta)}(L)}}\partial x_{\alpha\beta}}\left(\frac{\partial f_{(\alpha,\beta)}}{\partial x_{r_Lc_{\omega^{(\alpha,\beta)}(L)}}}\cdot\frac{\partial f_{(\alpha,\beta)}}{\partial x_{\alpha\beta}}-\frac{\partial f_{(\alpha,\beta)}}{\partial x_{\alpha c_{\omega^{(\alpha,\beta)}(L)}}}\cdot\frac{\partial f_{(\alpha,\beta)}}{\partial x_{r_L\beta}}\right)\\
&+\frac{\partial^2 f_{(\alpha,\beta)}}{\partial x_{r_Lc_{\omega^{(\alpha,\beta)}(L+1)}}\partial x_{\alpha\beta}}\left(\frac{\partial f_{(\alpha,\beta)}}{\partial x_{r_Lc_{\omega^{(\alpha,\beta)}(L+1)}}}\cdot\frac{\partial f_{(\alpha,\beta)}}{\partial x_{\alpha\beta}}-\frac{\partial f_{(\alpha,\beta)}}{\partial x_{\alpha c_{\omega^{(\alpha,\beta)}(L+1)}}}\cdot\frac{\partial f_{(\alpha,\beta)}}{\partial x_{r_L\beta}}\right).
\end{align*}
We recall from \cite[Section 4]{BCH} that if $\chi$ is the determinant of an $n\times n$ matrix of variables, then the matrix of its derivatives has rank $1$ on $\chi=0$. Therefore,
\begin{align*}
\left(\frac{\partial f_{(\alpha,\beta)}}{\partial x_{ij}}\right)_{i\in\mathcal{R}(\alpha,\beta)\cup\{\alpha\}, j\in\mathcal{C}(\alpha,\beta)\cup\{\beta\}}
\end{align*}
has rank $1$ at $\widehat{\omega}^{(\alpha,\beta)}(y_1,y_2,y_3)$. Hence, we deduce that
\begin{align*}
\text{Hess}f_{(\alpha,\beta)}\left(\overline{\nabla}f_{(\alpha,\beta)},\overline{\nabla}f_{(\alpha,\beta)}\right)=0
\end{align*}
at $\widehat{\omega}^{(\alpha,\beta)}(y_1,y_2,y_3)$.
\end{proof}

\subsection{Proof of the theorem}
To compute the trace of the Hessian of $f_{(\alpha,\beta)}$ in the normal space, we begin by considering the Gram matrix and its inverse with respect to the normal frame provided in Proposition \ref{Prop26}. We decompose the normal frame as 
\begin{align*}
\left\{\overline{\nabla}f_{(i,j)}\mid (i,j)\in\mathcal{D}(\omega)\right\}=\mathcal{B}_1\sqcup\mathcal{B}_2,
\end{align*}
where
\begin{align*}
&\mathcal{B}_1=\left\{\overline{\nabla}f_{(r_L,c_{\omega^{(\alpha,\beta)}(L+1)})}, \overline{\nabla}f_{(\alpha,\beta)}\right\},\\
&\mathcal{B}_2=\left\{\overline{\nabla}f_{(i,j)}\Bigm| (i,j)\in\mathcal{D}(\omega)\setminus\{(r_L,c_{\omega^{(\alpha,\beta)}(L+1)}),(\alpha,\beta)\}\right\}.
\end{align*}
It follows from Lemma \ref{Lem42} and \ref{Lem46} that $\mathcal{B}_1\perp\mathcal{B}_2$. 

On the other hand, by Lemma \ref{Lem46} we have
\begin{align*}
\overline{\nabla}f_{(r_L,c_{\omega^{(\alpha,\beta)}(L+1)})}&\left(\widehat{\omega}^{(\alpha,\beta)}(y_1,y_2,y_3)\right)\\
&=\frac{\partial f_{(r_L,c_{\omega^{(\alpha,\beta)}(L+1)})}}{\partial x_{r_Lc_{\omega^{(\alpha,\beta)}(L+1)}}}\left(\widehat{\omega}^{(\alpha,\beta)}(y_1,y_2,y_3)\right)\cdot E_{r_L,c_{\omega^{(\alpha,\beta)}(L+1)}}.
\end{align*}
Observe that $(r_L,c_{\omega^{(\alpha,\beta)}(L+1)})\in\mathcal{D}(\omega)$ cannot satisfy (\ref{star}) since $(\alpha,\beta)\in\mathcal{D}(\omega)$ is a one that satisfies (\ref{star}) and minimizes the sum $\alpha+\beta$. This implies that 
\begin{align*}
\frac{\partial f_{(r_L,c_{\omega^{(\alpha,\beta)}(L+1)})}}{\partial x_{r_Lc_{\omega^{(\alpha,\beta)}(L+1)}}}\left(\widehat{\omega}^{(\alpha,\beta)}(y_1,y_2,y_3)\right)=1,
\end{align*}
and therefore
\begin{align*}
\overline{\nabla}f_{(r_L,c_{\omega^{(\alpha,\beta)}(L+1)})}\left(\widehat{\omega}^{(\alpha,\beta)}(y_1,y_2,y_3)\right)=E_{r_L,c_{\omega^{(\alpha,\beta)}(L+1)}}.
\end{align*}
Moreover, Lemma \ref{Lem42} gives
\begin{align*}
\overline{\nabla}f_{(\alpha,\beta)}&\left(\widehat{\omega}^{(\alpha,\beta)}(y_1,y_2,y_3)\right)\\
&=\delta \cdot E_{\alpha,\beta}-y_1\delta\cdot E_{\alpha,c_{\omega^{(\alpha,\beta)}(L)}}+y_1y_3\delta\cdot E_{\alpha,c_{\omega^{(\alpha,\beta)}(L+1)}}\\
&-y_2\delta\cdot E_{r_L,\beta}+y_1y_2\delta\cdot E_{r_L,c_{\omega^{(\alpha,\beta)}(L)}}-y_1y_2y_3\delta\cdot E_{r_L,c_{\omega^{(\alpha,\beta)}(L+1)}},
\end{align*}
where $\delta=\det \left(\omega^{(\alpha,\beta)}\right)$. 

Putting these computations together, the Gram matrix $G$ with respect to $\mathcal{B}_1\sqcup\mathcal{B}_2$ at $\widehat{\omega}^{(\alpha,\beta)}(y_1,y_2,y_3)$ is block diagonal of the form
\begin{align*}
G=
\begin{pmatrix}
G_1 & O\\
O & G_2
\end{pmatrix},
\end{align*}
where each $O$ denotes a zero matrix of the appropriate size,
\begin{align*}
G_1=
\begin{pmatrix}
1 & -y_1y_2y_3\delta\\
-y_1y_2y_3\delta & (1+y_2^2)(1+y_1^2+y_1^2y_3^2)
\end{pmatrix},
\end{align*}
and $G_2$ is the Gram matrix corresponding to $\mathcal{B}_2$ at $\widehat{\omega}^{(\alpha,\beta)}(y_1,y_2,y_3)$. Consequently, its inverse is given by
\begin{align*}
G^{-1}=
\begin{pmatrix}
G_1^{-1} & O\\
O & G_2^{-1}
\end{pmatrix},
\end{align*}
where
\begin{align*}
G_1^{-1}=
\begin{pmatrix}
(1+y_2^2)(1+y_1^2+y_1^2y_3^2) & y_1y_2y_3\delta\\
y_1y_2y_3\delta & 1
\end{pmatrix}.
\end{align*}

We are now in a position to prove the theorem.
\begin{proof}[Proof of Theorem \ref{Thm41}]
Let $\omega\in\mathfrak{M}_{m,n}(\mathbb{R})$ be a non-vexillary partial permutation. Take $(\alpha,\beta)\in\mathcal{D}(\omega)$ and consider $\widehat{\omega}^{(\alpha,\beta)}(y_1,y_2,y_3)\in X_{\omega}$ as in Section \ref{subsec42}. Then, it follows from Lemma \ref{Lem47}, \ref{Lem48}, and \ref{Lem49}, together with the above computations on the Gram matrix $G$, that the trace of the Hessian of $f_{(\alpha,\beta)}$ in the normal space of $X_{\omega}$ at $\widehat{\omega}^{(\alpha,\beta)}(y_1,y_2,y_3)$ is given by
\begin{align*}
\text{tr}_{T^{\perp}X_{\omega}}\text{Hess}f_{(\alpha,\beta)}&\left(\widehat{\omega}^{(\alpha,\beta)}(y_1,y_2,y_3)\right)\\
&=2\left(G^{-1}\right)_{12}\text{Hess}f_{(\alpha,\beta)}\left(\overline{\nabla}f_{(\alpha,\beta)},\overline{\nabla}f_{(r_L, c_{\omega^{(\alpha,\beta)}(L+1)})}\right)\\
&=2\left(G_1^{-1}\right)_{12}\text{Hess}f_{(\alpha,\beta)}\left(\overline{\nabla}f_{(\alpha,\beta)},\overline{\nabla}f_{(r_L, c_{\omega^{(\alpha,\beta)}(L+1)})}\right)\\
&=2(-1)^{L+\omega^{(\alpha,\beta)}(L+1)}\delta y_1y_2y_3^2\\
&\neq 0,
\end{align*}
since $y_1,y_2,y_3$ are nonzero real numbers and $\delta=\det \left(\omega^{(\alpha,\beta)}\right)\neq 0$. Therefore, by Lemma \ref{minimalcheck}, the mean curvature vectors of $X_{\omega}$ along the points $\widehat{\omega}^{(\alpha,\beta)}(y_1,y_2,y_3)$ do not vanish. This completes the proof.
\end{proof}

\section{Minimality}\label{grass}
\setcounter{equation}{0}
In this section, we establish the minimality of real matrix Schubert varieties for certain subclasses of vexillary partial permutations. These subclasses include determinantal varieties and yield new examples of minimal submanifolds.

Let $\mathfrak{Gr}_2$ denote the set of vexillary partial permutations of the following forms:
\begin{align}\label{IOIOIO}
\begin{pmatrix}
I_r & O_1\\
O_2 & O_3
\end{pmatrix}\quad\text{or}\quad
\begin{pmatrix}
I_{r_1} & O_1 & O_2 & O_3\\
O_4 & O_5 & I_{r_2} & O_6\\
O_7 & O_8 & O_9 & O_{10}
\end{pmatrix}\quad\text{or}\quad
\begin{pmatrix}
I_{r_1} & O_1 & O_2\\
O_3 & O_4 & O_5\\
O_6 & I_{r_2} & O_7\\
O_8 & O_9 & O_{10}
\end{pmatrix},
\end{align}
where $I_r$ denotes the $r\times r$ identity matrix and each $O_j$ denotes a zero matrix of appropriate size. We further define $\widetilde{\mathfrak{Gr}}_2$ to be the set of partial permutations $\mu$ such that $\mathcal{D}(\mu)=\mathcal{D}(\omega)$ for some $\omega\in\mathfrak{Gr}_2$. Clearly, $\mathfrak{Gr}_2\subset\widetilde{\mathfrak{Gr}}_2$. As the following example shows, this inclusion is proper.
\begin{Ex}\label{ex51}\normalfont
Let $\omega$ and $\mu$ be partial permutations given by
\begin{align*}
\omega=
\begin{pmatrix}
1 & 0 & 0 & 0\\
0 & 0 & 1 & 0\\
0 & 0 & 0 & 0
\end{pmatrix},\quad
\mu=
\begin{pmatrix}
1 & 0 & 0 & 0 & 0\\
0 & 0 & 1 & 0 & 0\\
0 & 0 & 0 & 0 & 1\\
0 & 1 & 0 & 0 & 0
\end{pmatrix}.
\end{align*}
Then $\omega\in\mathfrak{Gr}_2$ and $\mathcal{D}(\omega)=\mathcal{D}(\mu)$. Hence $\mu\in\widetilde{\mathfrak{Gr}}_2$, while $\mu\notin\mathfrak{Gr}_2$. 
\end{Ex}
A permutation with only one descent is called \emph{Grassmannian}. One can observe that the Rothe diagrams of elements in $\widetilde{\mathfrak{Gr}}_2$ have the same shape as those of Grassmannian permutations (i.e., they are \emph{of Grassmannian type}) and have at most two connected components.

We first state the following theorem.
\begin{thm}\label{thm52}
Let $\omega\in\mathfrak{M}_{m,n}(\mathbb{R})$ be a vexillary partial permutation. If $\omega\in\widetilde{\mathfrak{Gr}}_2$, then $X_{\omega}\subset\mathfrak{M}_{m,n}(\mathbb{R})$ is minimal.
\end{thm}
Since the determinantal varieties correspond to partial permutations in $\mathfrak{Gr}_2$ (see Example \ref{ex27}), this theorem extends the result in \cite{BCH}. We also remark that the partial permutation $\omega$ given in Example \ref{ex51} provides the simplest non-determinantal example among the new minimal submanifolds established by the above theorem.
\begin{rmk}\label{rmk53}\normalfont
Let $\omega$ and $\mu$ be the partial permutations given in Example \ref{ex51}. As mentioned at the beginning of Section \ref{MSchubert}, the defining ideal of a real matrix Schubert variety depends only on the elements of the Rothe diagram and the corresponding upper-left submatrices of variables. Since $\mathcal{D}(\omega)=\mathcal{D}(\mu)$, the varieties $\overline{X}_{\omega}$ and $\overline{X}_{\mu}$ share the same defining equations. Consequently, the variables corresponding to the fourth row and the fifth column in $\overline{X}_{\mu}$ are free, and hence $\overline{X}_{\omega}\times\mathbb{R}^{8}\subset\mathbb{R}^{12}\times\mathbb{R}^{8}$ is congruent to $\overline{X}_{\mu}\subset\mathbb{R}^{20}$. Therefore, they are essentially equivalent from the viewpoint of minimality.
\end{rmk}

We now consider a vexillary partial permutation $\omega$ such that $(1,1)\in\mathcal{D}(\omega)$. The corresponding real matrix Schubert variety $\overline{X}_{\omega}$ can be decomposed as a product of smaller real matrix Schubert varieties and a Euclidean factor. 

For instance, let $\omega\in\mathfrak{M}_{7,8}(\mathbb{R})$ be given by
\begin{align}\label{ohmmm}
\omega=
\begin{pmatrix}
0 & 0 & 0 & 0 & 0 & 0 & 1 & 0\\
0 & 0 & 0 & 0 & 0 & 0 & 0 & 0\\
0 & 0 & 0 & 0 & 1 & 0 & 0 & 0\\
0 & 0 & 0 & 0 & 0 & 0 & 0 & 1\\
1 & 0 & 0 & 0 & 0 & 0 & 0 & 0\\
0 & 0 & 1 & 0 & 0 & 0 & 0 & 0\\
0 & 0 & 0 & 0 & 0 & 1 & 0 & 0
\end{pmatrix}.
\end{align}
The Rothe diagram is
\begin{align*}
\mathcal{D}(\omega)=&\left\{(1,1), (1,2), (1,3), (1,4), (1,5), (1,6), (2,1), (2,2), (2,3), (2,4), (2,5), (2,6),\right.\\
&\left. (3,1), (3,2), (3,3), (3,4), (4,1), (4,2), (4,3), (4,4)\right\}\\
&\cup\left\{(2,8)\right\}\cup\left\{(4,6)\right\}\cup\left\{(6,2),(7,2)\right\}\cup\left\{(7,4)\right\},
\end{align*}
which is decomposed into connected components. 

The defining equations coming from the first connected component are
\begin{align*}
&x_{11}=x_{12}=x_{13}=x_{14}=x_{15}=x_{16}=x_{21}=x_{22}=x_{23}=x_{24}=x_{25}=x_{26}\\
&=x_{31}=x_{32}=x_{33}=x_{34}=x_{41}=x_{42}=x_{43}=x_{44}=0.
\end{align*}
These define a subspace $P$ of dimension $36$ in $\mathfrak{M}_{7,8}(\mathbb{R})$ (codimension $20$), and we have $\overline{X}_{\omega}\subset P$. 

The equations from the second component correspond to the vanishing of all $2\times 2$ minors of
\begin{align*}
\mathbb{X}_{[2,8]}=
\begin{pmatrix}
x_{11} & x_{12} & x_{13} & x_{14} & x_{15} & x_{16} & x_{17} & x_{18}\\
x_{21} & x_{22} & x_{23} & x_{24} & x_{25} & x_{26} & x_{27} & x_{28}
\end{pmatrix},
\end{align*}
which, in the subspace $P$, are equivalent to 
\begin{align*}
\det
\begin{pmatrix}
x_{17} & x_{18}\\
x_{27} & x_{28}
\end{pmatrix}=0.
\end{align*}
This defines a real matrix Schubert variety in $\mathfrak{M}_{2,2}(\mathbb{R})$, with variables $x_{17}$, $x_{18}$, $x_{27}$, $x_{28}$, corresponding to
\begin{align*}
\omega_1:=\omega\big|_{\{1,2\},\{7,8\}}=
\begin{pmatrix}
1 & 0\\
0 & 0
\end{pmatrix}.
\end{align*}
Similarly, the equations from the third component are equivalent in $P$ to 
\begin{align*}
\det
\begin{pmatrix}
x_{35} & x_{36}\\
x_{45} & x_{46}
\end{pmatrix}=0,
\end{align*}
giving a real matrix Schubert variety in $\mathfrak{M}_{2,2}(\mathbb{R})$, with variables $x_{35}$, $x_{36}$, $x_{45}$, $x_{46}$, corresponding to
\begin{align*}
\omega_2:=\omega\big|_{\{3,4\},\{5,6\}}=
\begin{pmatrix}
1 & 0\\
0 & 0
\end{pmatrix}.
\end{align*}
Finally, the equations from the fourth and fifth components are equivalent in $P$ to the vanishing of all $2\times 2$ minors of 
\begin{align*}
\begin{pmatrix}
x_{51} & x_{52}\\
x_{61} & x_{62}\\
x_{71} & x_{72}
\end{pmatrix}
\end{align*}
and all $3\times 3$ minors of
\begin{align*}
\begin{pmatrix}
x_{51} & x_{52} & x_{53} & x_{54}\\
x_{61} & x_{62} & x_{63} & x_{64}\\
x_{71} & x_{72} & x_{73} & x_{74}
\end{pmatrix},
\end{align*}
defining a real matrix Schubert varieties in $\mathfrak{M}_{3,4}(\mathbb{R})$, with variables $x_{51}$, $x_{52}$, $x_{53}$, $x_{54}$, $x_{61}$, $x_{62}$, $x_{63}$, $x_{64}$, $x_{71}$, $x_{72}$, $x_{73}$, $x_{74}$, corresponding to
\begin{align*}
\omega_3:=\omega\big|_{\{5,6,7\},\{1,2,3,4\}}=
\begin{pmatrix}
1 & 0 & 0 & 0\\
0 & 0 & 1 & 0\\
0 & 0 & 0 & 0
\end{pmatrix}.
\end{align*}

The remaining variables $x_{37}$, $x_{38}$, $x_{47}$, $x_{48}$, $x_{55}$, $x_{56}$, $x_{57}$, $x_{58}$, $x_{65}$, $x_{66}$, $x_{67}$, $x_{68}$, $x_{75}$, $x_{76}$, $x_{77}$, $x_{78}$ are not contained in the upper-left submatrices corresponding to any elements of the Rothe diagram. Hence they are free variables, giving rise to a Euclidean factor $\mathbb{R}^{16}$. 

By grouping the defining equations in $P$ into algebraically independent subsets, one can deduce from the above observations that the real matrix Schubert variety $\overline{X}_{\omega}\subset P$ decomposes as
\begin{align*}
\overline{X}_{\omega}\cong\overline{X}_{\omega_1}\times\overline{X}_{\omega_2}\times\overline{X}_{\omega_3}\times \mathbb{R}^{16}. 
\end{align*}
The same type of decomposition can be easily seen to hold for any vexillary partial permutation $\omega$ satisfying $(1,1)\in\mathcal{D}(\omega)$.

Combining the above observation with Theorem \ref{thm52}, we immediately obtain the following corollary.
\begin{cor}\label{corcorcorcor}
Let $\omega\in\mathfrak{M}_{m,n}(\mathbb{R})$ be a vexillary partial permutation such that $(1,1)\in\mathcal{D}(\omega)$. Then $\overline{X}_{\omega}$ is congruent to $\overline{X}_{\omega_1}\times\cdots\times\overline{X}_{\omega_k}\times\mathbb{R}^N$ for some $N\geq 0$ and vexillary partial permutations $\omega_1,\cdots,\omega_k$. If all $\omega_j$ belong to $\widetilde{\mathfrak{Gr}}_2$, then $X_{\omega}\subset\mathfrak{M}_{m,n}(\mathbb{R})$ is minimal.
\end{cor}
By the above corollary, the partial permutation $\omega$ given in (\ref{ohmmm}) gives rise to a minimal submanifold in $\mathfrak{M}_{7,8}(\mathbb{R})$. Since its Rothe diagram has more than two connected components, $\omega$ does not belong to $\widetilde{\mathfrak{Gr}}_2$.

In the remaining of this section, we provide the proof of Theorem \ref{thm52}.

\subsection{Involutive isometries}\label{subsec51}
Motivated by the geometric proof given in \cite{BCH}, we consider involutive isometries. Let $\omega\in\mathfrak{M}_{m,n}(\mathbb{R})$ be a partial permutation of the form
\begin{align}\label{omegaaaa}
\omega=\begin{pmatrix}
I_{r_1} & O_{r_1\times n_1} & O_{r_1\times r_2} & O_{r_1\times n_2}\\
O_{r_2\times r_1} & O_{r_2\times n_1} & I_{r_2} & O_{r_2\times n_2}\\
O_{m_1\times r_1} & O_{m_1\times n_1} & O_{m_1\times r_2} & O_{m_1\times n_2}
\end{pmatrix}\in\mathfrak{Gr}_2,
\end{align}
where $m=m_1+r_1+r_2$ and $n=n_1+n_2+r_1+r_2$, and where $I_r$ and $O_{a\times b}$ denote the $r\times r$ identity and the $a\times b$ zero matrix, respectively.

Suppose $Q\in X_\omega$. Let us denote by $C(Q)$ the column space of $Q$, and let $U_Q\in O(m)$ be the unique element of the orthogonal group satisfying
\begin{align*}
U_Q\big|_{C(Q)}=id,\quad U_Q\big|_{C(Q)^{\perp}}=-id.
\end{align*}
Define $\Phi_Q: \mathfrak{M}_{m,n}(\mathbb{R})\to\mathfrak{M}_{m,n}(\mathbb{R})$ by
\begin{align*}
\Phi_Q(A):=\left(U_Q\cdot A_1\ \Big| \ \cdots\ \Big| \ U_Q\cdot A_n\right),\quad \forall A=\left(A_1\ \Big| \ \cdots\ \Big| \ A_n\right)\in\mathfrak{M}_{m,n}(\mathbb{R}).
\end{align*}
Since $U_Q^2=id$, we have $\Phi_Q\circ \Phi_Q=id$. As $U_Q\in O(m)$, it follows that the map $\Phi_Q$ is an involutive isometry of $\mathfrak{M}_{m,n}(\mathbb{R})$ fixing $Q$.
\begin{lemma}\label{lem55}
For each $Q\in X_\omega$, the map $\Phi_Q$ induces a symmetry of $\overline{X}_{\omega}$ such that $\Phi_Q(Q)=Q$.
\end{lemma}
\begin{proof}
It suffices to show that $\Phi_Q\left(X_\omega\right)\subset \overline{X}_{\omega}$. Then, since $X_\omega$ is open dense in $\overline{X}_\omega$, the continuity of $\Phi_Q$ implies that $\Phi_Q\left(\overline{X}_\omega\right)=\overline{X}_\omega$.

Suppose $A\in X_\omega$. To prove that $\Phi_Q(A)\in\overline{X}_\omega$, we verify 
\begin{align}\label{ineq1}
\text{rk}\left(\Phi_Q(A)_{[p,q]}\right)\leq \text{rk}\left(\omega_{[p,q]}\right)
\end{align}
for all $1\leq p\leq m$ and $1\leq q\leq n$ (see Section \ref{MSchubert} for the notation).
We divide all possible pairs of $p$ and $q$ into the following six cases:
\begin{itemize}
\item[(1)] $1\leq q<p\leq r_1$.
\item[(2)] $1\leq p\leq r_1$, $1\leq q\leq n$, $p\leq q$. 
\item[(3)] $r_1+1\leq p\leq r_1+r_2$, $r_1+n_1+1\leq q\leq n$, $p\leq q-n_1$.
\item[(4)] $r_1+1\leq p\leq m$, $1\leq q\leq r_1+n_1$.
\item[(5)] $r_1+r_2+1\leq p\leq m$, $r_1+n_1+1\leq q\leq n$.
\item[(6)] $r_1+1\leq q-n_1<p\leq r_1+r_2$.
\end{itemize}

In case $(1)$, since $\Phi_Q(A)_{[p,q]}$ has $q$ columns, its rank is at most $q$, which equals $\text{rk}\left(\omega_{[p,q]}\right)$.
Similarly, in cases $(2)$ and $(3)$, the matrix $\Phi_Q(A)_{[p,q]}$ has $p$ rows. Hence, its rank is at most $p=\text{rk}\left(\omega_{[p,q]}\right)$.

For cases $(4)$ and $(5)$, we observe that
\begin{align*}
\text{rk}\left(\Phi_Q(A)_{[p,q]}\right)&\leq \text{rk}\left(\Phi_Q(A)_{[m,q]}\right)\\
&=\text{rk}\left(U_Q\cdot A_1\ \Big| \ \cdots\ \Big| \ U_Q\cdot A_q\right)\\
&\leq\text{rk}\left(A_1\ \Big| \ \cdots\ \Big| \ A_q\right)=\text{rk}\left(A_{[m,q]}\right).
\end{align*}
Since $A\in X_\omega$, we have $\text{rk}\left(A_{[m,q]}\right)=\text{rk}\left(\omega_{[m,q]}\right)$. Moreover, $\text{rk}\left(\omega_{[m,q]}\right)=\text{rk}\left(\omega_{[p,q]}\right)$ holds for these cases. Thus, the inequality (\ref{ineq1}) follows.

Finally, in case $(6)$, we use the inequality
\begin{align*}
\text{rk}\left(\Phi_Q(A)_{[p,q]}\right)\leq \text{rk}\left(\Phi_Q(A)_{[p,r_1+n_1]}\right)+(q-r_1-n_1).
\end{align*}
By case $(4)$, the first term is bounded by $\text{rk}\left(\omega_{[p,r_1+n_1]}\right)=r_1$. A direct computation gives $r_1+(q-r_1-n_1)=\text{rk}\left(\omega_{[p,q]}\right)$, which yields the desired inequality. 

Thus, the inequality (\ref{ineq1}) holds in all cases. Since this holds for every $A\in X_\omega$, we obtain $\Phi_Q\left(X_\omega\right)\subset \overline{X}_{\omega}$. This completes the proof.
\end{proof}
\begin{rmk}\normalfont
We remark that the above map was first introduced in \cite{BCH} to show that the determinantal varieties have helicoidal symmetries. In our case the map $\Phi_Q$ does not necessarily induce a helicoidal symmetry.
\end{rmk}

Similarly, for $Q\in X_\omega$, let us denote by $R\left(Q_{[m,r_1+n_1]}\right)$ the row space of $Q_{[m,r_1+n_1]}$. Consider $V_Q\in O(r_1+n_1)$ the unique element of the orthogonal group satisfying  
\begin{align*}
V_Q\big|_{R\left(Q_{[m,r_1+n_1]}\right)}=id,\quad V_Q\big|_{R\left(Q_{[m,r_1+n_1]}\right)^{\perp}}=-id,
\end{align*}
and extend $V_Q$ to $\widehat{V}_Q\in O(n)$ by setting
\begin{align*}
\widehat{V}_Q\cdot
\begin{pmatrix}
v_1\\
v_2
\end{pmatrix}=
\begin{pmatrix}
V_Q\cdot v_1\\
v_2
\end{pmatrix},\quad \forall v_1\in\mathbb{R}^{r_1+n_1}, v_2\in\mathbb{R}^{r_2+n_2}.
\end{align*}
Define $\Psi_Q: \mathfrak{M}_{m,n}(\mathbb{R})\to\mathfrak{M}_{m,n}(\mathbb{R})$ by
\begin{align*}
\Psi_Q(B)=
\begin{pmatrix}
\widehat{V}_Q\cdot B_1\\
\vdots\\
\widehat{V}_Q\cdot B_m
\end{pmatrix},\quad \forall B=
\begin{pmatrix}
B_1\\
\vdots\\
B_m
\end{pmatrix}\in\mathfrak{M}_{m,n}(\mathbb{R}).
\end{align*}
Since $\widehat{V}_Q^2=id$ and $\widehat{V}_Q\in O(n)$, it follows that $\Psi_Q$ is an involutive isometry of $\mathfrak{M}_{m,n}(\mathbb{R})$ fixing $Q$.
\begin{lemma}\label{lem57}
For each $Q\in X_\omega$, the map $\Psi_Q$ induces a symmetry of $\overline{X}_{\omega}$ such that $\Psi_Q(Q)=Q$.
\end{lemma}
\begin{proof}
As in the proof of Lemma \ref{lem55}, we show $\Psi_Q\left(X_\omega\right)\subset \overline{X}_\omega$ by verifying that for every $B\in X_\omega$,
\begin{align}\label{ineq2}
\text{rk}\left(\Psi_Q(B)_{[p,q]}\right)\leq \text{rk}\left(\omega_{[p,q]}\right)
\end{align}
for all $1\leq p\leq m$ and $1\leq q\leq n$. All possible pairs of $p$ and $q$ can be divided into the following six cases:
\begin{itemize}
\item[(1)] $1\leq p\leq m$, $1\leq q\leq r_1$, $q<p$.
\item[(2)] $1\leq p\leq r_1$, $1\leq q\leq n$, $p\leq q$.
\item[(3)] $r_1+1\leq p\leq r_1+r_2$, $r_1+n_1+1\leq q\leq n$, $p\leq q-n_1$.
\item[(4)] $r_1+1\leq p\leq m$, $r_1+1\leq q\leq r_1+n_1$.
\item[(5)] $r_1+1\leq p\leq m$, $r_1+n_1+1\leq q\leq r_1+n_1+r_2$, $q-n_1<p$.
\item[(6)] $r_1+r_2+1\leq p\leq m$, $r_1+n_1+r_2+1\leq q\leq n$.
\end{itemize}

In case $(1)$, since $\Psi_Q(B)_{[p,q]}$ has $q$ columns, its rank is at most $q=\text{rk}\left(\omega_{[p,q]}\right)$. Similarly, in cases $(2)$ and $(3)$, the matrix $\Psi_Q(B)_{[p,q]}$ has $p$ rows so that its rank is at most $p$. In these cases, we have $p=\text{rk}\left(\omega_{[p,q]}\right)$, and the inequality (\ref{ineq2}) follows.

For case $(4)$, we have
\begin{align*}
\text{rk}\left(\Psi_Q(B)_{[p,q]}\right)\leq \text{rk}\left(\Psi_Q(B)_{[p,r_1+n_1]}\right)\leq \text{rk}\left(B_{[p,r_1+n_1]}\right)
\end{align*}
as in cases $(4)$ and $(5)$ in the proof of Lemma \ref{lem55}. Since $B\in X_\omega$, it follows that $\text{rk}\left(B_{[p,r_1+n_1]}\right)=\text{rk}\left(\omega_{[p,r_1+n_1]}\right)$. Moreover, $\text{rk}\left(\omega_{[p,r_1+n_1]}\right)=\text{rk}\left(\omega_{[p,q]}\right)$ holds, yielding the inequality (\ref{ineq2}).

In case $(5)$, we observe that
\begin{align*}
\text{rk}\left(\Psi_Q(B)_{[p,q]}\right)\leq\text{rk}\left(\Psi_Q(B)_{[p,r_1+n_1]}\right)+(q-r_1-n_1).
\end{align*}
By case $(4)$, the first term is bounded by $\text{rk}\left(\omega_{[p,r_1+n_1]}\right)=r_1$. A straightforward computation gives $r_1+(q-r_1-n_1)=\text{rk}\left(\omega_{[p,q]}\right)$. Hence, $\text{rk}\left(\Psi_Q(B)_{[p,q]}\right)\leq \text{rk}\left(\omega_{[p,q]}\right)$.

Finally, in case $(6)$, by applying an argument similar to that of case $(4)$, we have 
\begin{align*}
\text{rk}\left(\Psi_Q(B)_{[p,q]}\right)\leq \text{rk}\left(\Psi_Q(B)_{[p,n]}\right)\leq \text{rk}\left(B_{[p,n]}\right)=\text{rk}\left(\omega_{[p,n]}\right)=\text{rk}\left(\omega_{[p,q]}\right).
\end{align*}

Thus, the inequality (\ref{ineq2}) holds in all cases, and therefore $\Psi_Q(B)\in \overline{X}_\omega$. Since this holds for every $B\in X_\omega$, we obtain $\Psi_Q\left(X_\omega\right)\subset \overline{X}_\omega$. By continuity, it follows that $\Psi_Q\left(\overline{X}_\omega\right)=\overline{X}_\omega$, which completes the proof.
\end{proof}

\subsection{Action of the isometries on normal vectors}
We continue to consider the partial permutation $\omega$ given in (\ref{omegaaaa}). We modify the argument in \cite{BCH} to analyze the action of the isometries introduced in the previous subsection on normal vectors.

Let $Q\in X_\omega$. For $A\in\mathfrak{M}_{m,n}(\mathbb{R})$, we write
\begin{align*}
A=\left(A_1\ \Big| \ \cdots\ \Big| \ A_n\right),
\end{align*}
where each $A_j$ is a column vector. Define $t_c(A)\in\mathfrak{M}_{m,n}(\mathbb{R})$ by
\begin{align*}
t_c(A)=\left(O_{m\times (r_1+n_1)}\ \Big| \ \text{pr}_{C(Q)}\left(A_{r_1+n_1+1}\right)\ \Big| \ \cdots\ \Big| \ \text{pr}_{C(Q)}\left(A_{n}\right)\right),
\end{align*}
where $\text{pr}_{C(Q)}$ denotes the orthogonal projection onto the column space $C(Q)$ of $Q$, and $O_{m\times (r_1+n_1)}$ denotes the $m\times (r_1+n_1)$ zero matrix.
\begin{lemma}\label{lem58}
The vector $t_c(A)$ is tangent to $X_\omega$ at $Q$.
\end{lemma}
\begin{proof}
We prove the lemma by showing that $Q+s\cdot t_c(A)\in\overline{X}_\omega$ for every $s\in\mathbb{R}$. This implies that $Q+s\cdot t_c(A)\in X_\omega$ for all sufficiently small $s$, and hence $t_c(A)$ is tangent to $X_\omega$ at $Q$.

We check whether the following inequality holds:
\begin{align}\label{ineq3}
\text{rk}\left(\left(Q+s\cdot t_c(A)\right)_{[p,q]}\right)\leq \text{rk}\left(\omega_{[p,q]}\right)
\end{align}
for all $1\leq p\leq m$ and $1\leq q\leq n$. We divide all possible pairs of $p$ and $q$ into the following cases:
\begin{itemize}
\item[(1)] $1\leq p\leq r_1$, $r_1+n_1+1\leq q\leq n$.
\item[(2)] $r_1+1\leq p\leq r_1+r_2$, $r_1+n_1+1\leq q\leq n$, $p\leq q-n_1$.
\item[(3)] $1\leq p\leq m$, $1\leq q\leq r_1+n_1$.
\item[(4)] $r_1+1\leq p\leq m$, $r_1+n_1+1\leq q\leq r_1+n_1+r_2$, $q-n_1<p$.
\item[(5)] $r_1+r_2+1\leq p\leq m$, $r_1+n_1+r_2+1\leq q\leq n$.
\end{itemize}

In cases $(1)$ and $(2)$, the matrix $\left(Q+s\cdot t_c(A)\right)_{[p,q]}$ has rank at most $p$ since it has $p$ rows. As $\text{rk}\left(\omega_{[p,q]}\right)=p$ in these cases, the inequality (\ref{ineq3}) follows.

For case $(3)$, we have $t_c(A)_{[p,q]}=O$. Hence,
\begin{align*}
\text{rk}\left(\left(Q+s\cdot t_c(A)\right)_{[p,q]}\right)=\text{rk}\left(Q_{[p,q]}\right)=\text{rk}\left(\omega_{[p,q]}\right).
\end{align*}

In case $(4)$, we observe that
\begin{align*}
\text{rk}\left(\left(Q+s\cdot t_c(A)\right)_{[p,q]}\right)&\leq \text{rk}\left(\left(Q+s\cdot t_c(A)\right)_{[p,r_1+n_1]}\right)+(q-r_1-n_1)\\
&=\text{rk}\left(\omega_{[p,r_1+n_1]}\right)+(q-r_1-n_1),
\end{align*}
where the equality follows from case $(3)$. A direct computation gives 
\begin{align*}
\text{rk}\left(\omega_{[p,r_1+n_1]}\right)+(q-r_1-n_1)=q-n_1=\text{rk}\left(\omega_{[p,q]}\right),
\end{align*}
which confirms the inequality (\ref{ineq3}).

Finally, in case $(5)$, we have
\begin{align*}
\text{rk}\left(\left(Q+s\cdot t_c(A)\right)_{[p,q]}\right)\leq \text{rk}\left(\left(Q+s\cdot t_c(A)\right)_{[m,q]}\right).
\end{align*}
Since all columns of $\left(Q+s\cdot t_c(A)\right)_{[m,q]}$ lie in the column space $C(Q)$, it follows that
\begin{align*}
\text{rk}\left(\left(Q+s\cdot t_c(A)\right)_{[m,q]}\right)\leq \text{rk}(Q).
\end{align*}
Given that $Q\in X_\omega$, we have
\begin{align*}
\text{rk}(Q)=\text{rk}\left(Q_{[m,n]}\right)=\text{rk}\left(\omega_{[m,n]}\right)=r_1+r_2.
\end{align*}
As $\text{rk}\left(\omega_{[p,q]}\right)=r_1+r_2$ in this case, the inequality (\ref{ineq3}) holds.

Thus, (\ref{ineq3}) holds in all cases, and therefore $Q+s\cdot t_c(A)\in\overline{X}_\omega$ for every $s\in\mathbb{R}$. This completes the proof.
\end{proof}
\begin{lemma}\label{lem59}
Let $N\in\mathfrak{M}_{m,n}(\mathbb{R})$ be a normal vector to $X_\omega$ at $Q$, written as 
\begin{align*}
N=\left(N_1\ \Big| \ \cdots\ \Big| \ N_n\right).
\end{align*}
Then the isometry $\Phi_Q$ reverses the sign of $N_{r_1+n_1+1},\cdots, N_n$. That is,
\begin{align*}
\left(\Phi_Q\right)_{*}N=\left(U_Q\cdot N_1\ \Big| \ \cdots\ \Big| \ U_Q\cdot N_{r_1+n_1}\ \Big| \ -N_{r_1+n_1+1}\ \Big| \ -N_n\right).
\end{align*}
\end{lemma}
\begin{proof}
Since $t_c(N)$ is a tangent vector by Lemma \ref{lem58}, we have
\begin{align*}
0=\left\langle N, t_c(N)\right\rangle=\sum_{j=r_1+n_1+1}^{n}\left|\text{pr}_{C(Q)}\left(N_j\right)\right|^2,
\end{align*}
which implies that $N_j\in C(Q)^{\perp}$ for all $r_1+n_1+1\leq j\leq n$. It then follows from the definition of $U_Q$ that
\begin{align*}
U_Q\cdot N_j=-N_j
\end{align*}
for all $r_1+n_1+1\leq j\leq n$, which completes the proof.
\end{proof}
In a similar manner, we introduce another tangent vector associated with the row spaces. Consider a matrix $B\in\mathfrak{M}_{m,n}(\mathbb{R})$ of the form
\begin{align}\label{Bform}
B=
\begin{pmatrix}
B_1 & O\\
\vdots & \vdots\\
B_m & O
\end{pmatrix},
\end{align}
where each $B_j$ is a row vector in $\mathbb{R}^{r_1+n_1}$, and $O$ denotes the zero vector in $\mathbb{R}^{r_2+n_2}$. Since $Q\in X_\omega$, the row space $R\left(Q_{[r_1,r_1+n_1]}\right)$ of $Q_{[r_1,r_1+n_1]}$ and the row space $R\left(Q_{[r_1,n]}\right)$ of $Q_{[r_1,n]}$ are of full rank. Using the rows of $Q_{[r_1,r_1+n_1]}$ and $Q_{[r_1,n]}$ as bases for their respective row spaces, we can associate to each $v\in R\left(Q_{[m,r_1+n_1]}\right)=R\left(Q_{[r_1,r_1+n_1]}\right)$ a unique row vector $f(v)\in\mathbb{R}^{r_2+n_2}$ such that $\left(v\ f(v)\right)\in R\left(Q_{[r_1,n]}\right)$. 

Define $t_r(B)\in\mathfrak{M}_{m,n}(\mathbb{R})$ by
\begin{align*}
t_r(B)=
\begin{pmatrix}
\text{pr}_{R\left(Q_{[m,r_1+n_1]}\right)}(B_1) & f\left(\text{pr}_{R\left(Q_{[m,r_1+n_1]}\right)}(B_1)\right)\\
\vdots & \vdots\\
\text{pr}_{R\left(Q_{[m,r_1+n_1]}\right)}(B_n) & f\left(\text{pr}_{R\left(Q_{[m,r_1+n_1]}\right)}(B_n)\right)
\end{pmatrix},
\end{align*}
where $\text{pr}_{R\left(Q_{[m,r_1+n_1]}\right)}$ is the orthogonal projection onto the row space $R\left(Q_{[m,r_1+n_1]}\right)$.
\begin{lemma}\label{lem510}
The vector $t_r(B)$ is tangent to $X_\omega$ at $Q$.
\end{lemma}
\begin{proof}
As in the proof of Lemma \ref{lem58}, it suffices to prove that for every $s\in\mathbb{R}$,
\begin{align}\label{ineq4}
\text{rk}\left(\left(Q+s\cdot t_r(B)\right)_{[p,q]}\right)\leq \text{rk}\left(\omega_{[p,q]}\right)
\end{align}
for all $1\leq p\leq m$ and $1\leq q\leq n$. All possible pairs of $p$ and $q$ can be divided into the following cases:
\begin{itemize}
\item[(1)] $1\leq p\leq r_1$, $1\leq q\leq n$, $p\leq q$.
\item[(2)] $r_1+1\leq p\leq r_1+r_2$, $r_1+n_1+1\leq q\leq n$, $p\leq q-n_1$.
\item[(3)] $1\leq p\leq m$, $1\leq q\leq r_1$, $q<p$.
\item[(4)] $r_1+1\leq p\leq m$, $r_1+1\leq q\leq r_1+n_1$.
\item[(5)] $r_1+1\leq p\leq m$, $r_1+n_1+1\leq q\leq r_1+n_1+r_2$, $q-n_1<p$.
\item[(6)] $r_1+r_2+1\leq p\leq m$, $r_1+n_1+r_2+1\leq q\leq n$.
\end{itemize}

In cases $(1)$ and $(2)$, the matrix $\left(Q+s\cdot t_r(B)\right)_{[p,q]}$ has $p$ rows so that its rank is bounded above by $p=\text{rk}\left(\omega_{[p,q]}\right)$. Likewise, in case $(3)$, the matrix $\left(Q+s\cdot t_r(B)\right)_{[p,q]}$ has $q$ columns, which implies that its rank is at most $q=\text{rk}\left(\omega_{[p,q]}\right)$.

In case $(4)$, we have
\begin{align*}
\text{rk}\left(\left(Q+s\cdot t_r(B)\right)_{[p,q]}\right)\leq\text{rk}\left(\left(Q+s\cdot t_r(B)\right)_{[p,r_1+n_1]}\right). 
\end{align*}
Since $Q\in X_\omega$, the two row spaces $R\left(Q_{[p,r_1+n_1]}\right)$ and $R\left(Q_{[m,r_1+n_1]}\right)$ coincide. Consequently, all rows of $\left(Q+s\cdot t_r(B)\right)_{[p,r_1+n_1]}$ lie in $R\left(Q_{[m,r_1+n_1]}\right)$, and hence
\begin{align*}
\text{rk}\left(\left(Q+s\cdot t_r(B)\right)_{[p,r_1+n_1]}\right)\leq \text{rk}\left(Q_{[m,r_1+n_1]}\right)=r_1=\text{rk}\left(\omega_{[p,q]}\right).
\end{align*}
This yields the inequality (\ref{ineq4}).

In case $(5)$, we obtain
\begin{align*}
\text{rk}\left(\left(Q+s\cdot t_r(B)\right)_{[p,q]}\right)&\leq\text{rk}\left(\left(Q+s\cdot t_r(B)\right)_{[p,r_1+n_1]}\right)+(q-r_1-n_1)\\
&\leq q-n_1=\text{rk}\left(\omega_{[p,q]}\right),
\end{align*}
where the second inequality follows from case $(4)$.

Finally, in case $(6)$, an argument similar to that of case $(4)$ yields (\ref{ineq4}).

Therefore, the inequality (\ref{ineq4}) holds in all cases, completing the proof.
\end{proof}
\begin{lemma}\label{lem511}
Suppose that a normal vector $N$ to $X_\omega$ at $Q$ has the form given in (\ref{Bform}). Then the isometry $\Psi_Q$ satisfies
\begin{align*}
\left(\Psi_Q\right)_{*}N=-N.
\end{align*}
\end{lemma}
\begin{proof}
We write 
\begin{align*}
N=
\begin{pmatrix}
N_1 & O\\
\vdots & \vdots\\
N_m & O
\end{pmatrix}
\end{align*}
as in (\ref{Bform}). Since $t_r(N)$ is tangent to $X_\omega$ at $Q$ by Lemma \ref{lem510}, we obtain
\begin{align*}
0=\left\langle N, t_r(N)\right\rangle=\sum_{j=1}^{m}\left|\text{pr}_{R\left(Q_{[m,r_1+n_1]}\right)}\left(N_j\right)\right|^2.
\end{align*}
This implies that $N_j\in R\left(Q_{[m,r_1+n_1]}\right)^{\perp}$ for all $1\leq j\leq m$. It then follows from the definition of $V_Q$ that
\begin{align*}
\left(\Psi_Q\right)_{*}N=
\begin{pmatrix}
V_Q\cdot N_1 & O\\
\vdots & \vdots\\
V_Q\cdot N_m & O
\end{pmatrix}=
\begin{pmatrix}
-N_1 & O\\
\vdots & \vdots\\
-N_m & O
\end{pmatrix}=-N.
\end{align*}
Therefore the lemma is proved.
\end{proof}

\subsection{Proof of the theorem}
We now state the proof of Theorem \ref{thm52}.
\begin{proof}[Proof of Theorem \ref{thm52}]
As observed in Remark \ref{rmk53}, if two partial permutations have the same Rothe diagram, then their corresponding real matrix Schubert varieties differ only by a Euclidean factor. Hence, the minimality of one implies that of the other, and vice versa. Therefore, it suffices to prove the theorem for the class $\mathfrak{Gr}_2$.

Let $\omega\in\mathfrak{M}_{m,n}(\mathbb{R})$ be a vexillary partial permutation in $\mathfrak{Gr}_2$. The minimality of determinantal varieties, which correspond to the first form in (\ref{IOIOIO}) (see Example \ref{ex27}), was established in \cite{BCH}. By taking transposes, the second and third forms in (\ref{IOIOIO}) are equivalent. Thus, we may assume that $\omega$ is given as in (\ref{omegaaaa}).

Let $Q\in X_\omega$, and let $H(Q)\in\mathfrak{M}_{m,n}(\mathbb{R})$ denote the mean curvature vector of $X_\omega$ at $Q$, written as
\begin{align*}
H(Q)=\left(H(Q)_1\ \Big| \ \cdots\ \Big| \ H(Q)_n\right).
\end{align*}

Consider the isometry $\Phi_Q$ defined in Section \ref{subsec51}. By Lemma \ref{lem55}, we have
\begin{align*}
\left(\Phi_Q\right)_{*}H(Q)=H(Q).
\end{align*}
On the other hand, applying Lemma \ref{lem59} to $H(Q)$ gives
\begin{align*}
\left(\Phi_Q\right)_{*}H(Q)=\left(U_Q\cdot H(Q)_1\ \Big| \ \cdots\ \Big| \ U_Q\cdot H(Q)_{r_1+n_1}\ \Big| \ -H(Q)_{r_1+n_1+1}\ \Big| \ -H(Q)_n\right).
\end{align*}
Comparing these two equalities, we deduce that $H(Q)_j=0$ for all $r_1+n_1+1\leq j\leq n$. This implies that $H(Q)$ is of the form (\ref{Bform}).

Next, consider the isometry $\Psi_Q$ defined in Section \ref{subsec51}. By Lemma \ref{lem57}, we have
\begin{align*}
\left(\Psi_Q\right)_{*}H(Q)=H(Q).
\end{align*}
It follows from Lemma \ref{lem511} that 
\begin{align*}
\left(\Psi_Q\right)_{*}H(Q)=-H(Q).
\end{align*}
Therefore, $H(Q)$ must vanish. Since this holds for every $Q\in X_\omega$, we conclude that $X_{\omega}\subset\mathfrak{M}_{m,n}(\mathbb{R})$ is minimal.
\end{proof}


\end{document}